\documentclass[11pt]{amsart}
\usepackage{amssymb}

\newtheorem{theorem}{Theorem}[section]
\newtheorem{corollary}[theorem]{Corollary}
\newtheorem{lemma}[theorem]{Lemma}

\theoremstyle{definition}
\newtheorem{definition}[theorem]{Definition}

\newtheorem{question}[theorem]{Question}
\theoremstyle{remark}
\newtheorem{remark}[theorem]{Remark}
\newtheorem{example}[theorem]{Example}

\numberwithin{equation}{section}

\def\R {{\mathbb{R}}}

\def\N{{\mathbb{N}}}

\def\Z {{\mathbb{Z}}}

\def\3{{|\!|\!|}}
\def\coef{{\mathrm{coef}}}

\begin{document}

\title{On equivalence relations generated by Schauder bases}
\author{Longyun Ding}
\address{School of Mathematical Sciences and LPMC, Nankai University, Tianjin, 300071, P.R.China}
\email{dinglongyun@gmail.com}
\thanks{Research partially supported by the National Natural Science Foundation of China (Grant No. 11371203).}

\subjclass[2010]{Primary 03E15, 46B15, 46B45}

\date{\today}

\begin{abstract}
In this paper, a notion of Schauder equivalence relation $\R^\N/L$ is introduced, where $L$ is a linear subspace of $\R^\N$ and the unit vectors form a Schauder basis of $L$. The main theorem is to show that the following conditions are equivalent:
\begin{enumerate}
\item[(1)] the unit vector basis is boundedly complete;
\item[(2)] $L$ is $F_\sigma$ in $\R^\N$;
\item[(3)] $\R^\N/L$ is Borel reducible to $\R^\N/\ell_\infty$.
\end{enumerate}

We show that any Schauder equivalence relation generalized by basis of $\ell_2$ is Borel bireducible to $\R^\N/\ell_2$ itself, but it is not true for bases of $c_0$ or $\ell_1$. Furthermore, among all Schauder equivalence relations generated by sequences in $c_0$, we find the minimum and the maximum elements with respect to Borel reducibility.

We also show that $\R^\N/\ell_p$ is Borel reducible to $\R^\N/J$ iff $p\le 2$, where $J$ is James' space.
\end{abstract}
\maketitle

\section{Introduction}

The notion of Borel reducibility becomes a tool to compare objects or problems from different branches of mathematics. In recent years, many equivalence relations concerning Banach space theory were investigated. One motivation of this paper is the Borel reducibility among equivalence relations $\R^\N/\ell_p$ and $\R^\N/c_0$. It is proved by Dougherty and Hjorth: for $p,q\ge 1$,
$$\R^\N/\ell_p\le_B\R^\N/\ell_q\iff p\le q,$$
while $\R^\N/\ell_p$ and $\R^\N/c_0$ are Borel incomparable (see \cite{DH} and \cite{hjorth}). These results were generalized via different methods by several authors (see, e.g., \cite{matrai} and \cite{ding2}). In this paper, we study equivalence relations of the form $\R^\N/L$, where $L$ is a linear subspace of $\R^\N$. Moreover, the class of of subspaces $L$ we focus on in this paper can be specified by the following two equivalent conditions:
\begin{enumerate}
\item[(i)] there is a sequence $(x_n)$ of none-zero elements in a Banach space $X$ such that
$$L=\coef(X,(x_n))\stackrel{\rm Def}{=}\{a\in\R^\N:\sum_na(n)x_n\mbox{ converges in }X\}.$$
\item[(ii)] the unit vectors $e_n=(0,0,\cdots,0,\stackrel{n}{1},0,\cdots)$ form a Schauder basis of $L$.
\end{enumerate}
If one, and thus all of above conditions hold for $L$, we call $\R^\N/L$ a Schauder equivalence relation.

Schauder equivalence relations were already studied in different disguises. For example, it is obvious that all $\R^\N/\ell_p\;(p\ge 1)$ and $\R^\N/c_0$ are Schauder equivalence relations. Most recently, equivalence relations $\R^{\N\times\N}/\ell_p(\ell_q)$ were considered by Gao and Yin \cite{GY}. We can easily see that, for any $p,q\ge 1$, $\R^{\N\times\N}/\ell_p(\ell_q)$ is Borel bireducible to a Schauder equivalence relation. With a continuous function $f:[0,1]\to\R^+$, M\'atrai \cite{matrai} defined a relation ${\bf E}_f$ on $[0,1]^\N$. Borel reducibility between equivalence relations of the form ${\bf E}_f$ were investigated in \cite{matrai}. Yin \cite{yin} generalized M\'atrai's results to show that the partial order structure $P(\omega)/{\rm Fin}$ can be embedded into the set of these ${\bf E}_f$'s equipped with the partial ordering of Borel reducibility. In fact, as noted in Yin \cite{yin}, any such ${\bf E}_f$ appeared in \cite{yin} is Borel bireducible to a Schauder equivalence relation $\R^\N/L$, where $L$ is an Orlicz sequence space.

The main theorem of this paper is the following:

\begin{theorem}\label{main}
If the unit vectors $(e_n)$ form a Schauder basis of a Banach space $L$. Then the following are equivalent:
\begin{enumerate}
\item[(1)] $(e_n)$ is boundedly complete basis;
\item[(2)] $L$ is $F_\sigma$ in $\R^\N$;
\item[(3)] $\R^\N/L\le_B\R^\N/\ell_\infty$.
\end{enumerate}
\end{theorem}

In is well known that $\R^\N/\ell_p\sim_B[0,1]^\N/\ell_p$ and $\R^\N/c_0\sim_B[0,1]^\N/c_0$. We generalize these results to the following:
\begin{theorem}\label{cube1}
Let $(x_n)$ be a symmetric basis of Banach space $X$. Then
$$E(X,(x_n))\sim_B[0,1]^\N/\coef(X,(x_n)).$$
\end{theorem}

Reducibility and nonreducibility between Schauder equivalence relations generated by different sequences of same space are considered, especially in case that the generating sequences are Schauder bases. In this paper, we mainly focus on bases of three special Banach spaces: $\ell_2$, $c_0$, and $\ell_1$. 

\begin{theorem}
For any basis $(y_k)$ of $\ell_2$, we have $E(\ell_2,(y_k))\sim_B\R^\N/\ell_2$.
\end{theorem}

In contract, for $c_0$, we construct special bases $(x^m_n)$ for each $m\ge 1$ and $m=\infty$, and denote ${\rm cs}^{(m)}=\coef(c_0,(x^m_n))$. For $m=1$, we have
$${\rm cs}^{(1)}=\{a\in\R^\N:\sum_na(n)\mbox{ converges}\}.$$
We show that
\begin{theorem}
\begin{enumerate}
\item[(1)] For $m\ge 1$, we have
$$\R^\N/c_0<_B\R^\N/{\rm cs}^{(m)}\le_B\R^\N/{\rm cs}^{(m+1)}<_B\R^\N/{\rm cs}^{(\infty)}.$$
\item[(2)] Let $(x_n)$ be a none-zero sequence in $c_0$. Then
$$\R^\N/c_0\le_BE(c_0,(x_n))\le_B\R^\N/{\rm cs}^{(\infty)}.$$
\end{enumerate}
\end{theorem}
While for $\ell_1$, we construct a basis $(y^1_n)$ with
$$\coef(\ell_1,(y^1_n))={\rm bv}_0\stackrel{\rm Def}{=}c_0\cap\{a\in\R^\N:\sum_n|a(n)-a(n+1)|<+\infty\},$$
and prove $\R^\N/\ell_1<_B\R^\N/{\rm bv}_0<_B\R^\N/\ell_1\otimes\R^\N/c_0$,
where $\otimes$ is the direct product operator between equivalence relations.

We also compare $\R^\N/\ell_p$ and $\R^\N/J$ where $J$ is James' space and get
$$\R^\N/\ell_p\le_B\R^\N/J\iff p\le 2.$$

The paper is organized as follows. In section 2 we recall some notions in descriptive set theory and functional analysis, and introduce two kind of equivalence relaitons. In section 3 we prove Theorem \ref{main}. In section 4 we prove an useful lemma for converting a Borel reduction to an additive reduction. In section 5 we focus on Schauder equivalence relations generated by bases of $\ell_2$, $c_0$, and $\ell_1$. In section 6 we prove Theorem \ref{cube1} and compare $\R^\N/\ell_p$ and $\R^\N/J$. Finally section 7 contains some further remarks and open problems.

\section{Preliminaries}

A {\it Polish space} is a separable completely metrizable topological space. Let $E$ and $F$ be equivalence relations on Polish spaces $X$ and $Y$ respectively. A Borel function $\theta:X\to Y$ is called a {\it Borel reduction} from $E$ to $F$ if, for any $x,y\in X$,
$$(x,y)\in E\iff(\theta(x),\theta(y))\in F.$$
We say $E$ is {\it Borel reducible} to $F$, denoted $E\le_B F$, if there exists a Borel reduction from $E$ to $F$. If both $E\le_BF$, $F\le_BE$ hold, we say $E$ and $F$ are {\it Borel bireducible}, denoted $E\sim_B F$. We also denote $E\le_B F$ and $F\not\le_B E$ as $E<_B F$. We refer to \cite{gaobook} and \cite{kanovei} for background of Borel reducibility.

A sequence $(x_n)$ in a Banach space $X$ is called a {\it Schauder basis} (or basis, for the sake of brevity) of $X$ if, for any $x\in X$, there is a unique sequence $a\in\R^\N$ such that $x=\sum_na(n)x_n$. Let $(x_n)$ be a Schauder basis of $X$. Define $P_n:X\to X$ as $P_n(\sum_na(n)x_n)=\sum_{i\le n}a(n)x_n$. Then all $P_n$ are bounded and the {\it basis constant} $\sup_n\|P_n\|<+\infty$. It follows that, there is a sequence $(x^*_n)$ of bounded linear functional on $X$, such that $x=\sum_nx^*_n(x)x_n$. We call $(x_n^*)$ the {\it biorthogonal functionals} associated to $(x_n)$.

Let $(x_n)$ be a basis of $X$. We say $(x_n)$ is {\it unconditional} if, for any permutation $\pi:\N\to\N$, the sequence $(x_{\pi(n)})$ is also a basis. A basis $(x_n)$ is said to be {\it boundedly complete} if, for every sequence $a\in\R^\N$ such that $\sup_n\|\sum_{i\le n} a(i)x_i\|<\infty$, the series $\sum_na(n)x_n$ converges.

Let $(x_n)$ be a sequence of none-zero elements in a Banach space $X$. The closed linear span of $\{x_n:n\in\N\}$ is denoted by $[x_n]_{n\in\N}$. We denote
$$\coef(X,(x_n))=\{a\in\R^\N:\sum_na(n)x_n\mbox{ converges}\},$$
and for $a\in\coef(X,(x_n))$, we define
$$\3 a\3=\sup_n\|\sum_{i\le n}a(i)x_i\|.$$
By Cauchy criterion, it is routine to check that $(\coef(X,(x_n)),\3\cdot\3)$ is a Banach space, and the unit vectors $e_n=(0,0,\cdots,0,\stackrel{n}{1},0,\cdots)$ form a basis of it. From the definition of $\coef(X,(x_n))$, we can easily see that, if $X$ is a closed subspace of $Y$, then $\coef(X,(x_n))=\coef(Y,(x_n))$.

A sequence $(x_n)$ is called {\it normalized} if $\|x_n\|=1$ for each $n$, and is called {\it semi-normalized} if there are $A\ge B>0$ such that $A\ge\|x_n\|\ge B$ for each $n$. It is easy too see that, for a semi-normalized sequence $(x_n)$ of $X$, we always have $\ell_1\subseteq\coef(X,(x_n))\subseteq c_0$.

We say two sequences $(x_n)$ and $(y_n)$ of $X$ are {\it equivalent} if $\coef(X,(x_n))=\coef(X,(y_n))$. A basis $(x_n)$ of $X$ is said to be {\it symmetric} if, for any permutation $\pi:\N\to\N$, $(x_{\pi(n)})$ is equivalent to $(x_n)$. All symmetric basis are actually unconditional.

\begin{definition}
Let $(x_n)$ be a sequence in a Banach space $X$. We define an equivalence relation on $\R^\N$ as $E(X,(x_n))=\R^\N/\coef(X,(x_n))$, i.e., for all $a,b\in\R^\N$,
$$(a,b)\in E(X,(x_n))\iff a-b\in\coef(X,(x_n)).$$
We call this kind of equivalence relations {\it Schauder equivalence relations}.
\end{definition}

A none-zero sequence $(x_n)$ of a Banach space $X$ is said to be a {\it basic sequence} if it is a basis of $[x_n]_{n\in\N}$.

Let $(x_n)$ be a basis of $X$, $(r_n)$ a sequence of real numbers, and $0=n_0<n_1<\cdots$ an strictly increasing natural numbers. If for every $k$, $u_k=\sum_{n=n_k}^{n_{k+1}-1}r_nx_n$ is not $0$, we call sequence $(u_k)$ a {\it block basis} of $(x_n)$. A block basis is no necessarily a basis, but is always a basic sequence. A simple reduction $\theta$ witnesses that $E(X,(u_k))\le_BE(X,(x_n))$ defined as, for any $a\in\R^\N$ and $n\in\N$, $\theta(a)(n)=a(k)r_n$ for $n_k\le n<n_{k+1}$. For any sequence $(r_n)$ of non-zero numbers, we always have $E(X,(x_n))\sim_BE(X,(r_nx_n))$. Therefore, we may assume any basis is normalized when we need.

Let $(x_n)$ be a semi-normalized basis of $X$, $(r_n)$ a sequence of positive real numbers with $\sum_nr_n<+\infty$. Since $\ell_1\subseteq\coef(X,(x_n))$, we have
$$E(X,(x_n))\sim(\prod_nr_n\Z)/\coef(X,(x_n)).$$
A desired Borel reduction $\theta$ defined as $\theta(a)(n)=r_n[a(n)/r_n]$ for $a\in\R^\N$ and $n\in\N$.

Let $E$ and $F$ be two equivalence relations on $X$ and $Y$ respectively. Denote $E\otimes F$ the equivalence relation on $X\times Y$ as
$$((x_1,y_1),(x_2,y_2))\in E\otimes F\iff(x_1,x_2)\in E\;\&\;(y_1,y_2)\in F$$
for $x_1,x_2\in X$ and $y_1,y_2\in Y$. Let $(x_n)$ be a sequence in $X$ and $(y_n)$ a sequence in $Y$. We denote $z_{2k}=x_k$ and $z_{2k+1}=y_k$ for $k\in\N$. It is easy to see that
$$E(X\oplus Y,(z_n))\sim_B E(X,(x_n))\otimes E(Y,(y_n)).$$
So we may think $E(X,(x_n))\otimes E(Y,(y_n))$ is still a Schauder equivalence relation. Furthermore, if $(x_n)$ and $(y_n)$ are Schauder bases on $X$ and $Y$ respectively, we can see that $(z_n)$ is also a basis on $X\oplus Y$.

A sequence $(X_n)$ of closed subspaces of a Banach space $X$ is called a {\it Schauder decomposition} of $X$ if every $x\in X$ has a unique representation of the form $x=\sum_nx_n$, with $x_n\in X_n$ for each $n$. Similar to Schauder bases, define $P_n:X\to X$ as $P_n(\sum_nx_n)=\sum_{i\le n}x_i$ where $x_n\in X_n$ for each $n$. Then all $P_n$ are bounded and the {\it decomposition constant} $\sup_n\|P_n\|<+\infty$.

\begin{definition}
Let $(X_n)$ be a Schauder decomposition of a separable Banach space $X$. We define an equivalence relation $E(X,(X_n))$ on $\prod_nX_n$ as
$$(\alpha,\beta)\in E(X,(X_n))\iff\sum_n(\alpha(n)-\beta(n))\mbox{ converges in }X$$
for any $\alpha,\beta\in\prod_nX_n$. We call this kind of equivalence relations {\it decomposition equivalence relations}.

Furthermore, if all these $X_n$ are finite dimensional, we call $E(X,(X_n))$ an {\it F.D.D. equivalence relation}.
\end{definition}

For more details for Schauder bases and Schauder decompositions, we refer to \cite{LT}. A tiny difference on notation with \cite{LT} is, in this paper, any sequence $(x_n)$ means $(x_0,x_1,\cdots)$, not $(x_1,x_2,\cdots)$.

\section{$F_\sigma$ Schauder equivalence relations}

In the light of Rosendal's Theorem that any $K_\sigma$ equivalence relation on a Polish space is Borel reducible to $\R^\N/\ell_\infty$ (see \cite{rosendal}), we compare $F_\sigma$ Schauder equivalence relations and $\R^\N/\ell_\infty$.

The following lemma will be used to convert a Borel reduction to a continuous reduction.

\begin{lemma}\label{useful1}
Denote $D=\{d\in\R^\N:\forall n(4^nd(n)\in\Z)\}$. Let $G$ be a dense $G_\delta$ set in $D$, $a\in\R^\N$ with $2^na(n)\in\Z$ for each $n\in\N$, and let $-1=n_0<n_1<\cdots<n_k<\cdots$. Then there exist $b^*\in D$ and a strictly increasing sequence of natural numbers $(k_l)$ with $k_0=0$ such that $G\supseteq C$, where
$$C=\{d\in D:\forall l\exists i\le 2^l\forall n\in(n_{k_l},n_{k_{l+1}}](d(n)=b^*(n)+ia(n)/2^l)\}.$$
\end{lemma}

\begin{proof}
Assume that $O_0\supseteq O_1\supseteq\cdots\supseteq O_l\supseteq\cdots$ be a sequence of dense open sets with $G=\bigcap_{l}O_l$. We will construct $b^*$ and $(k_m)$ by induction on $m$ such that $O_m\supseteq C_m$, where
$$C_m=\{d\in D:\forall l\le m\exists i\le 2^l\forall n\in(n_{k_l},n_{k_{l+1}}](d(n)=b^*(n)+ia(n)/2^l)\}.$$
When we finish the construction, we shall have
$$C=\bigcap_{m}C_m\subseteq\bigcap_mO_m=G.$$

First, for $m=0$, fix a $d^0_0\in O_0$. Since $O_0$ is open, there is $n^0_0$ such that $O_0\supseteq\{d\in D:\forall n\le n^0_0(d(n)=d^0_0(n))\}$. Set $b^*(n)=d^0_0(n)$ for $n\le n^0_0$. Denote
$$N^0_1=\{d\in D:\forall n\le n^0_0(d(n)=d^0_0(n)+a(n))\}.$$
Since $O_0$ is dense, there is a $d^0_1\in N^0_1\cap O_0$. Then we can find an $n^0_1\ge n^0_0$ such that $O_0\supseteq\{d\in D:\forall n\le n^0_1(d(n)=d^0_1(n))\}$.
Select $k_1$ such that $n_{k_1}\ge n^0_1$ and set $b^*(n)=d^0_1(n)-a(n)$ for $n^0_0<n\le n_{k_1}$. Now we denote
$$C_0=\{d\in D:\exists i\in\{0,1\}\forall n\le n_{k_1}(d(n)=b^*(n)+ia(n))\}.$$
Then $C_0\subseteq O_0$.

Secondly, assume that we have defined $k_1,\cdots,k_m$ and the value of $b^*(n)$ for $n\le n_{k_m}$. Let $s_0,s_1,\cdots,s_J$ be an enumeration of the following set
$$\{(i_0,\cdots,i_m):\forall l\le m(0\le i_l\le 2^l)\}.$$
We inductively find a sequence $n^m_0<n^m_1<\cdots<n^m_J$ as follows.

Denote
$$N^m_0=\{d\in D:\forall l<m\forall n\in(n_{k_l},n_{k_{l+1}}](d(n)=b^*(n)+s_0(l)a(n)/2^l)\}.$$
Since $O_m$ is dense, there is $d^m_0\in N^m_0\cap O_m$. Then since $O_m$ is open, we can find an $n^m_0\ge n_{k_m}$ such that $O_m\supseteq\{d\in D:\forall n\le n^m_0(d(n)=d^m_0(n))$. Set $b^*(n)=d^m_0(n)-s_0(m)a(n)/2^m$ for $n_{k_m}<n\le n^m_0$.

Further assume that $n^m_j$ and the value of $b^*(n)$ for $n\le n^m_j$ have been defined.

If $j<J$, denote by $N^m_{j+1}$ the set of all $d\in D$ satisfying
$$d(n)=\left\{\begin{array}{ll}b^*(n)+s_{j+1}(l)a(n)/2^l, & n\in(n_{k_l},n_{k_{l+1}}]\mbox{ for }l<m,\cr
b^*(n)+s_{j+1}(m)a(n)/2^m, & n_{k_m}<n\le n^m_j\end{array}\right.$$
By the same reason, we can find $d^m_{j+1}\in N^m_{j+1}\cap O_m$ and $n^m_{j+1}\ge n^m_j$ with $O_m\supseteq\{d\in D:\forall n\le n^m_j(d(n)=d^m_{j+1}(n))\}$. Then set $b^*(n)=d^m_{j+1}-s_{j+1}(m)a(n)/2^m$ for $n^m_j<n\le n^m_{j+1}$.

If $j=J$, select a $k_{m+1}$ such that $n_{k_{m+1}}\ge n^m_J$. Set $b^*(n)=d^m_J(n)-s_J(m)a(n)/2^m$ for $n^m_J<n\le n_{k_{m+1}}$. It easy to see $O_m\supseteq C_m$ as desired.
\end{proof}

The next theorem is slightly more general then Theorem \ref{main}.

Recall that a Schauder decomposition $(X_n)$ of a Banach space $X$ is called {\it boundedly complete} if, for every sequence $\alpha\in\prod_nX_n$ such that $\sup_n\|\sum_{i\le n} \alpha(i)\|<+\infty$, the series $\sum_n\alpha(n)$ converges (see \cite{LT}).

\begin{theorem}\label{decomposition}
Let $(X_n)$ be a Schauder decomposition of a separable Banach space $X$. Then the following are equivalent:
\begin{enumerate}
\item[(1)] $(X_n)$ is boundedly complete;
\item[(2)] ${\rm cs}(X,(X_n))\stackrel{\rm Def}{=}\{\alpha\in\prod_nX_n:\sum_n\alpha(n)\mbox{ converges}\}$ is $F_\sigma$ in $\prod_nX_n$;
\item[(3)] $E(X,(X_n))\le_B\R^\N/\ell_\infty$.
\end{enumerate}
\end{theorem}

\begin{proof} Define $S:X\to\prod_nX_n$ as $S(x)=(x_n)$ for $x=\sum_nx_n$ with each $x_n\in X_n$. Because all projections $P_n$ of the decomposition $(X_n)$ are bounded, we see $S$ is a continuous injection whose range is ${\rm cs}(X,(X_n))$. Let $M$ be the decomposition constant $\sup_n\|P_n\|$. For $m\ge 1$, we denote
$$\begin{array}{ll}B_m&=\{\alpha\in\prod_iX_i:\sup_n\|\sum_{i\le n}\alpha(i)\|\le m\}\cr &=\bigcap_n\{\alpha\in\prod_iX_i:\|\sum_{i\le n}\alpha(i)\|\le m\}.\end{array}$$
Then $B_m$ is closed in $\prod_nX_n$.

(1)$\Rightarrow$(2). From the definition of boundedly complete decomposition, we have
$${\rm cs}(X,(X_n))=\bigcup_mB_m,$$
so ${\rm cs}(X,(X_n))$ is $F_\sigma$.

(2)$\Rightarrow$(1). Assume that ${\rm cs}(X,(X_n))=\bigcup_mF_m$ with each $F_m$ closed in $\prod_nX_n$. Then each $S^{-1}(F_m)$ is closed in $X$. By Baire category theorem, there exits an $m$ such that $S^{-1}(F_m)$ has an inner point. So there exist $y^\#\in X$ and $r>0$ such that
$$S^{-1}(F_m)\supseteq B(y^\#,r)=\{x\in X:\|x-y^\#\|\le r\}.$$

Now for any sequence $\alpha\in\prod_nX_n$ with $\sup_n\|\sum_{i\le n}\alpha(i)\|<+\infty$, without loss of generality, we may assume that $\sup_n\|\sum_{i\le n}\alpha(i)\|\le r$. For each $j\in\N$, we define $\alpha_j\in\prod_nX_n$ as
$$\alpha_j(n)=\left\{\begin{array}{ll}\alpha(n),& n\le j,\cr 0,& n>j.\end{array}\right.$$
Then $\alpha_j\in{\rm cs}(X,(X_n))$ and $\|S^{-1}(\alpha_j)\|\le r$. Therefore, for $j\in\N$, we have $y^\#+S^{-1}(\alpha_j)\in B(b_0,r)\subseteq S^{-1}(F_m)$, i.e., $S(y^\#)+\alpha_j\in F_m$. Note that, in $\prod_nX_n$, $\lim_{j\to\infty}\alpha_j\to\alpha$. Since $F_m$ is closed in $\prod_nX_n$, we have $S(y^\#)+\alpha\in F_m\subseteq{\rm cs}(X,(X_n))$. Hence $\alpha\in{\rm cs}(X,(X_n))$.

(1)$\Rightarrow$(3). Let $\{U_k:k\in\N\}$ be a basis for the topology of $\prod_nX_n$. For $\alpha\in\prod_nX_n$ and $k\in\N$, since $\bigcup_{m}(\alpha+B_m)=\alpha+{\rm cs}(X,(X_n))$ is dense in $\prod_nX_n$, there are some $m$ such that $(\alpha+B_m)\cap U_k\ne\emptyset$. So we can define
$$\theta(\alpha)(k)=\min\{m:(\alpha+B_m)\cap U_k\ne\emptyset\}.$$
It is easy to see that $\theta:\prod_nX_n\to\R^\N$ is Borel.

For $\alpha,\beta\in\prod_nX_n$, if $(\alpha,\beta)\in E(X,(X_n))$, let $\|S^{-1}(\alpha-\beta)\|\le K/M$ with $K\in\N$. Then $\alpha-\beta\in B_K$, so for each $m$, $\alpha+B_m\subseteq\beta+B_{m+K}$, and $\beta+B_m\subseteq\alpha+B_{m+K}$. It follows that $|\theta(\alpha)(k)-\theta(\beta)(k)|\le K$ for each $k$, and hence $\theta(\alpha)-\theta(\beta)\in\ell_\infty$. On the other hand, if $(\alpha,\beta)\notin E(X,(X_n))$, then
$$(\alpha+{\rm cs}(X,(X_n)))\cap(\beta+{\rm cs}(X,(X_n)))=\emptyset.$$
Thus for each $l\ge 1$, we have $(\alpha+B_1)\cap\bigcup_{m\le l}(\beta+B_m)=\emptyset$. We can find a $k$ such that $\alpha+B_1$ meets $U_k$, and $(\beta+B_m)\cap U_k=\emptyset$ for all $m\le l$, so $|\theta(\alpha)(k)-\theta(\beta)(k)|\ge l$. It follows that $\theta(\alpha)-\theta(\beta)\notin\ell_\infty$.

Therefore, $\theta$ is a Borel reduction of $E(X,(X_n))$ to $\R^\N/\ell_\infty$.

(3)$\Rightarrow$(1). Assume for contradiction that there exists an $\alpha\in\prod_nX_n$ with $\sup_n\|\sum_{i\le n}\alpha(i)\|<+\infty$, but $\alpha\notin{\rm cs}(X,(X_n))$. By Cauchy criterion, there exist an $\varepsilon_0>0$ and a sequence $-1=n_0<n_1<\cdots<n_k<\cdots$ such that $\|\sum_{n=n_k+1}^{n_{k+1}}\alpha(n)\|\ge\varepsilon_0$. Denote $D=\{d\in\R^\N:\forall n(4^nd(n)\in\Z)\}$. Without loss of generality, we may assume that $\alpha(n)\ne 0$ for each $n$, otherwise, we may replace $\alpha$ by $\alpha+\gamma$ for a suitable $\gamma\in{\rm cs}(X,(X_n))$. Define $T:D\to\prod_nX_n$ as $T(d)(n)=d(n)\alpha(n)$ for any $d\in D$ and $n\in\N$. It is clear that $T$ is a homeomorphic embedding.

Suppose $\theta$ is a Borel reduction of $E(X,(X_n))$ to $\R^\N/\ell_\infty$. Then $\theta\circ T$ is also Borel on $D$. There exists a dense $G_\delta$ subset $G\subseteq D$ such that $\theta\circ T$ is continuous on $G$ (cf. \cite[(8.38)]{kechris}). Then $\theta$ is also continuous on $T(G)$. Applying Lemma \ref{useful1} with $a(n)=1$ for each $n$, There exist $b^*\in D$ and a strictly increasing sequence $(k_l)$ with $k_0=0$ such that $G\supseteq C$, where
$$C=\{d\in D:\forall l\exists i\le 2^l\forall n\in(n_{k_l},n_{k_{l+1}}](d(n)=b^*(n)+i/2^l)\}.$$

Let $C^*=T(C)\cap(T(b^*)+{\rm cs}(X,(X_n)))$. Note that $T(C)$ is closed in $\prod_nX_n$, so $C^*$ is relatively closed in $T(b^*)+{\rm cs}(X,(X_n))$. By the definition of Borel reduction, $C^*=(\theta\upharpoonright T(C))^{-1}(\theta(T(b^*))+\ell_\infty)$. So $C^*$ is $F_\sigma$ in $T(C)$, since $\theta$ is continuous on $T(C)$ and $\ell_\infty$ is $F_\sigma$ in $\R^\N$. Thus we can assume that $C^*=\bigcup_mF_m$ with each $F_m$ closed in $T(C)$.

Now denote $Z=S^{-1}(C^*-T(b^*)),Z_m=S^{-1}(F_m-T(b^*))$. Then $Z$ is closed in $X$, thus complete. Because each $Z_m$ is closed in $Z$ with $\bigcup_mZ_m=Z$, there exists an $m$ such that $Z_m$ has an inner point in $Z$. Thus there exist $y^\#\in Z$ and $r>0$ such that
$$Z_m\supseteq W=\{x\in Z:\|x-y^\#\|\le r\}.$$
Let $y^\#=\sum_ny_n$ with $y_n\in X_n$, by Cauchy criterion, we have
$$\lim_{k\to\infty}\|\sum_{n=n_k+1}^{n_{k+1}}y_n\|=0.$$
Since $S(y^\#)\in(C^*-T(b^*))\subseteq(T(C)-T(b^*))$, we have $T^{-1}(S(y^\#))\in(C-b^*)$. So there is a sequence $(i_l)$ such that, for $n_{k_l}<n\le n_{k_{l+1}}$, $T^{-1}(S(y^\#))(n)=i_l/2^l$, i.e., $y_n=S(y^\#)(n)=(i_l/2^l)\alpha(n)$. Thus
$$\lim_{l\to\infty}\frac{i_l}{2^l}\|\sum_{n=n_{k_l}+1}^{n_{k_l+1}}\alpha(n)\|=0.$$
Comparing with $\|\sum_{n=n_{k_l}+1}^{n_{k_l+1}}\alpha(n)\|\ge\varepsilon_0$, we get $\lim_{l\to\infty}i_l/2^l=0$. Now fix a large enough natural number $L$ such that $i_l/2^l<1/2$ for $l\ge L$, and
$$\frac{1}{2^L}\sup_n\|\sum_{i\le n}\alpha(i)\|\le\frac{r}{2}.$$
We define
$$\alpha'(n)=\left\{\begin{array}{ll}0, & n\le n_{k_L},\cr \alpha(n), & n>n_{k_L},\end{array}\right.$$
and for each $j>L$,
$$\alpha'_j(n)=\left\{\begin{array}{ll}0, & n\le n_{k_L}\mbox{ or }n>n_{k_j},\cr \alpha(n), & n_{k_L}<n\le n_{k_j}.\end{array}\right.$$
It is clear that $T(b^*)+S(y^\#)+a'/2^L\in T(C)$ and $T(b^*)+S(y^\#)+a'_j/2^L\in C^*$ for each $j$. Note that $S^{-1}(\alpha'_j/2^L)=1/2^L\sum_{n=n_{k_L}+1}^{n_{k_j}}\alpha(n)$ and
$$\frac{1}{2^L}\|\sum_{n=n_{k_L}+1}^{n_{k_j}}\alpha(n)\|\le\frac{1}{2^L}\left(\|\sum_{i\le n_{k_L}}\alpha(i)\|+\|\sum_{i\le n_{k_j}}\alpha(i)\|\right)\le r.$$
It follows that $y^\#+S^{-1}(\alpha'_j/2^L)\in W\subseteq Z_m=S^{-1}(F_m-T(b^*))$, i.e., $T(b^*)+S(y^\#)+\alpha'_j/2^L\in F_m$. Since $F_m$ is closed in $T(C)$, we have
$$T(b^*)+S(y^\#)+\alpha'/2^L=\lim_{j\to\infty}(T(b^*)+S(y^\#)+\alpha'_j/2^L)\in F_m\subseteq C^*.$$
From the definition of $C^*$, we have $S(y^\#)+\alpha'/2^L\in{\rm cs}(X,(X_n))$, so $\alpha'\in{\rm cs}(X,(X_n))$. Hence $\alpha\in{\rm cs}(X,(X_n))$, a contradiction!
\end{proof}

Indeed, the proof of (3)$\Rightarrow$(1) shows that, if $(X_n)$ is not boundedly complete, then $E(X,(X_n))$ is not Borel reducible to any $F_\sigma$ equivalence relation.

Theorem \ref{main} is equivalent to the following corollary.

\begin{corollary}\label{fsigma}
Let $(x_n)$ be a Schauder basis of a Banach space $X$. Then the following are equivalent:
\begin{enumerate}
\item[(1)] $(x_n)$ is boundedly complete;
\item[(2)] $\coef(X,(x_n))$ is $F_\sigma$ in $\R^\N$;
\item[(3)] $E(X,(x_n))\le_B\R^\N/\ell_\infty$.
\end{enumerate}
\end{corollary}

\section{One lemma on additive reductions}

A lemma for converting a Borel reduction to an additive reduction will be used again and again in the rest of this paper, especially for proving nonreducibility. We introduce some concerned notions first.

\begin{definition}[Farah \cite{farah2}]
\begin{enumerate}
\item[(a)] A map $\psi:\prod_nX_n\to\prod_nX_n'$ is {\it additive} if there are $0=l_0<l_1<\cdots<l_j<\cdots$ and maps $H_j:X_j\to\prod_{n\in[l_i,l_{i+1})}X_i'$ such that
$$\psi(\alpha)=H_0(\alpha(0))^\smallfrown H_1(\alpha(1))^\smallfrown H_2(\alpha(2))^\smallfrown\cdots.$$
\item[(b)] Let $E$ and $F$ be equivalence relations on $\prod_nX_n$ and $\prod_nX_n'$ respectively, we say $E$ is {\it additively reducible} to $F$, denoted $E\le_A F$, if there is an additive reduction of $E$ to $F$.
\end{enumerate}
\end{definition}

Let $E$ be an equivalence relation on $\prod_nX_n$, and let $I\subseteq\N$ be infinite. Fix an element $\mu\in\prod_{n\notin I}X_n$. For $\alpha\in\prod_{n\in I}X_n$, put $\alpha\oplus\mu=\left\{\begin{array}{ll}\alpha(n), & n\in I,\cr \mu(n), & n\notin I.\end{array}\right.$
We define $E|_I$ (with $\mu$) on $\prod_{n\in I}X_n$ as
$$(\alpha,\beta)\in E|_I\iff(\alpha\oplus\mu, \beta\oplus\mu)\in E.$$
If for any $\alpha_1,\alpha_2\in\prod_nX_n$, $(\alpha_1,\alpha_2)\in E$ is equivalent to say $\alpha_1-\alpha_2$ is in a specified set, then the exact value of $\mu(n)$ is meaningless, thus we may assume, say, $\mu(n)=0$ for all $n\notin I$, if we need.

Let $(F_n)$ be a sequence of finite sets. A special equivalence relation $E_0(\prod_nF_n)$ defined as
$$(\alpha,\beta)\in E_0(\prod_nF_n)\iff\exists m\forall n>m(\alpha(n)=\beta(n))$$
for all $\alpha,\beta\in\prod_nF_n$.

A weak version of the following lemma is due to Dougherty and Hjorth \cite{DH}.

\begin{lemma}\label{useful2}
Let $E$ be an equivalence relation on $\prod_nF_n$ with $E_0(\prod_nF_n)\subseteq E$, where all $F_n$ are finite sets. Let $(X_n)$ be a Schauder decomposition of a separable Banach space $X$. If $E\le_B E(X,(X_n))$, then there exists an infinite $I\subseteq\N$ such that $E|_I\le_AE(X,(X_n))$.
\end{lemma}

\begin{proof}
The proof is modified from the proof of \cite{DH}, Theorem 2.2, claims (i)--(iii). We omit some similar arguments.

Assume that $\theta$ is a Borel reduction of $E$ to $E(X,(X_n))$. Following claims (i) and (ii), and the arguments after Claim (ii) of \cite{DH}, Theorem 2.2, we construct two sequences of natural numbers $n_0<n_1<n_2<\cdots$ and $l_0<l_1<l_2<\cdots$, a sequence of $(s_j)$, and dense open sets $D^j_i\subseteq\prod_nF_n\;(i,j\in\N)$.

Denote $I=\{n_j:j\in\N\}$ and $\mu=\bigcup_js_j$. Note that ${\rm dom}(\mu)=\bigcup_j{\rm dom}(s_j)=\N\setminus I$. The construction confirms that, for any $\alpha,\hat\alpha\in\prod_{n\in I}F_n$, we have
\begin{enumerate}
\item[(a)] if $\alpha(n)=\hat\alpha(n)$ for $n>n_j$, then
$$\|\sum_{n\ge l_{j+1}}(\theta(\alpha\oplus\mu)(n)-\theta(\hat\alpha\oplus\mu)(n))\|<2^{-j};$$
\item[(b)] if $\alpha(n)=\hat\alpha(n)$ for $n\le n_j$, then
$$\|\sum_{n<l_{j+1}}(\theta(\alpha\oplus\mu)(n)-\theta(\hat\alpha\oplus\mu)(n))\|<2^{-j}.$$
\end{enumerate}

Now we are ready to define a sequence of mappings $(H_n)_{n\in I}$ to assemble the desired additive reduction $\psi$. For each $i\in I$, fix an element $x^*_n\in F_n$. We define $p_j:F_{n_j}\to\prod_{n\in I}F_n$ for each $j\in\N$ as
$p_j(x)(n)=\left\{\begin{array}{ll}x, & n=n_j,\cr x^*_n, & n\ne n_j\end{array}\right.$ for $x\in F_n$ and $n\in\N$. Then for $n\in I$ with $n=n_j$, we define $H_n:F_n\to\prod_{n\in[l_j,l_{j+1})}X_n$ as, for $x\in F_n$,
$$H_n(x)=\theta(p_j(x)\oplus\mu)\upharpoonright[l_j,l_{j+1}).$$
The additive mapping $\psi:\prod_{n\in I}F_n\to\prod_nX_n$ defined as, for $\alpha\in\prod_{n\in I}F_n$,
$$\psi(\alpha)=H_{n_0}(\alpha(n_0))^\smallfrown H_{n_1}(\alpha(n_1))^\smallfrown H_{n_2}(\alpha(n_2))^\smallfrown\cdots.$$

We come to show that $\psi$ is a reduction of $E|_I$ to $E(X,(X_n))$.

For any $\alpha\in\prod_{n\in I}F_n$ and $j\in\N$, define $e_j(\alpha),e'_j(\alpha)\in\prod_{n\in I}F_n$ as
$$e_j(\alpha)(n)=\left\{\begin{array}{ll}\alpha(n), & n=n_j,\cr x^*_n, & n\ne n_j,\end{array}\right.\quad
e'_j(\alpha)(n)=\left\{\begin{array}{ll}\alpha(n), & n\le n_j,\cr x^*_n, & n>n_j.\end{array}\right.$$
Applying (a) for $j-1$ and (b) for $j$, we have
$$\|\sum_{n\ge l_j}(\theta(e_j(\alpha)\oplus\mu)(n)-\theta(e'_j(\alpha)\oplus\mu)(n))\|<2^{-(j-1)},$$
$$\|\sum_{n<l_{j+1}}(\theta(\alpha\oplus\mu)(n)-\theta(e'_j(\alpha)\oplus\mu)(n))\|<2^{-j}.$$
Let $M$ be the decompositon constant of $(X_n)$. We have
$$\begin{array}{ll}&\|\sum_{n\in[l_j,l_{j+1})}(\theta(\alpha\oplus\mu)(n)-\theta(e_j(\alpha)\oplus\mu)(n))\|\cr
\le&\|\sum_{n\in[l_j,l_{j+1})}(\theta(\alpha\oplus\mu)(n)-\theta(e'_j(\alpha)\oplus\mu)(n))\|\cr
&+\|\sum_{n\in[l_j,l_{j+1})}(\theta(e'_j(\alpha)\oplus\mu)(n))-\theta(e_j(\alpha)\oplus\mu)(n))\|\cr
\le&(1+M)\|\sum_{n<l_{j+1}}(\theta(\alpha\oplus\mu)(n)-\theta(e'_j(\alpha)\oplus\mu)(n))\|\cr
&+M\|\sum_{n\ge l_j}(\theta(e'_j(\alpha)\oplus\mu)(n))-\theta(e_j(\alpha)\oplus\mu)(n))\|\cr
<&(1+3M)2^{-j}.
\end{array}$$
Note that $p_j(\alpha(n_j))=e_j(\alpha)$. We have
$$\psi(\alpha)\upharpoonright[l_j,l_{j+1})=H_{n_j}(\alpha(n_j))=\theta(e_j(\alpha)\oplus\mu)\upharpoonright[l_j,l_{j+1})$$
for each $j\in\N$. For each $m\in\N$, let $n_i\le m<n_{i+1}$. We have
$$\begin{array}{ll}&\|\sum_{n\ge m}(\theta(\alpha\oplus\mu)(n)-\psi(\alpha)(n))\|\cr
\le&\|\sum_{n\in[m,l_{i+1})}(\theta(\alpha\oplus\mu)(n)-\theta(e_i(\alpha)\oplus\mu)(n))\|\cr
&+\sum_{j\ge i+1}\|\sum_{n\in[l_j,l_{j+1})}(\theta(\alpha\oplus\mu)(n)-\theta(e_j(\alpha)\oplus\mu)(n))\|\cr
<&(1+M)(1+3M)2^{-i}+\sum_{j\ge i+1}(1+3M)2^{-j}\cr
=&(2+M)(1+3M)2^{-i}.
\end{array}$$
It follows that
$$\lim_{m\to\infty}\sum_{n\ge m}(\theta(\alpha\oplus\mu)(n)-\psi(\alpha)(n))=0,$$
i.e., $(\theta(\alpha\oplus\mu),\psi(\alpha))\in E(X,(X_n))$.

In the end, for $\alpha,\hat\alpha\in\prod_{n\in I}F_n$, we have
$$\begin{array}{ll}(\psi(\alpha),\psi(\hat\alpha))\in E(X,(X_n))&\iff(\theta(\alpha\oplus\mu),\theta(\hat\alpha\oplus\mu))\in E(X,(X_n))\cr
&\iff(\alpha\oplus\mu,\hat\alpha\oplus\mu)\in E\cr
&\iff(\alpha,\hat\alpha)\in E|_I.\end{array}$$
This completes the proof.
\end{proof}

\begin{remark}
It worth noting that the preceding lemma can be applied on some variations. If there is a subset $M_n\subseteq X_n$ for each $n$ such that $E\le_BE(X,(X_n))\upharpoonright\prod_nM_n$, i.e., there is a Borel reduction $\theta$ from $\prod_nF_n$ to $\prod_nM_n$, then the resulted additive reduction $\psi$ is also mapping into $\prod_nM_n$. Thus $E|_I\le_AE(X,(X_n))\upharpoonright\prod_nM_n$.
\end{remark}

As an application, we prove the following theorem.

\begin{theorem}\label{unconditional}
Let $(x_n)$ and $(y_n)$ be bases of Banach spaces $X$ and $Y$ respectively. If $(x_n)$ is unconditional, and every subsequence of $(y_n)$ is conditional, then $E(Y,(y_n))\not\le_BE(X,(x_n))$.
\end{theorem}

\begin{proof}
Assume for contradiction that $\theta:\R^\N\to\R^\N$ is a Borel reduction of $E(Y,(y_n))$ to $E(X,(x_n))$. Denote $F_n=\{i/2^n:i=0,1,\cdots,2^n\}$, and denote by $E$ the restriction of $E(Y,(y_n))$ on $\prod_nF_n$. Then $\theta\upharpoonright\prod_nF_n$ is also a Borel reduction of $E$ to $E(X,(x_n))$. From Lemma \ref{useful2}, we can find an infinite set $I\subseteq\N$ and an element $\mu\in\prod_{n\notin I}F_n$ such that $E|_I$ (with $\mu$) is additively reducible to $E(X,(x_n))$. Without loss of generality, we may assume that $\mu(n)=0$ for $n\notin I$. Therefore, we can find, for each $n\in I$, a natural number $l_n\ge 1$ and a map $H_n:F_n\to\R^{l_n}$ such that the following $\psi$ is a reduction of $E|_I$ to $E(X,(x_n))$. Let $(n_k)$ is the strictly increasing enumeration of $I$, then $\psi$ is defined as
$$\psi(a)=H_{n_0}(a(n_0))^\smallfrown H_{n_1}(a(n_1))^\smallfrown H_{n_2}(a(n_2))^\smallfrown\cdots,$$
for any $a\in\prod_{n\in I}F_n$.

From the assumption of $(y_n)$, the subsequence $(y_{n_k})$ is conditional. By Definition 1.c.5 and Proposition 1.c.6 of \cite{LT}, there exists an $a_0\in\R^\N$ such that $\sum_ka_0(k)y_{n_k}$ converges conditionally. Then by Proposition 1.c.1 of \cite{LT}, there is an $\epsilon\in\{-1,1\}^\N$ such that $\sum_ka_0(k)y_{n_k}$ converges while $\sum_k\epsilon(k)a_0(k)y_{n_k}$ diverges.

Without loss of generality, we may assume that $(y_n)$ is normalized. Then we have $\ell_1\subseteq\coef(Y,(y_{n_k}))\subseteq c_0$, so we may further assume that $a_0(k)\in F_{n_k}$. For each $n_k\in I$, denote $a^I(n_k)=a_0(k),a^\emptyset(n)=0$, and define
$$a^+(n_k)=\left\{\begin{array}{ll}a_0(k), & \epsilon(k)=1,\cr 0, & \epsilon(k)=-1,\end{array}\right.\quad
a^-(n_k)=\left\{\begin{array}{ll}0, & \epsilon(k)=1,\cr a_0(k), & \epsilon(k)=-1.\end{array}\right.$$
Since $\sum_{n\in I}(a^I(n)-a^\emptyset(n))y_n$ converges and $\sum_{n\in I}(a^+(n)-a^-(n))y_n$ diverges, we have $\psi(a^I)-\psi(a^\emptyset)\in\coef(X,(x_n))$ and $\psi(a^+)-\psi(a^-)\notin\coef(X,(x_n))$. Denote $t_k=H_{n_k}(a_0(k))-H_{n_k}(0)$. Then we have
$$t_0{}^\smallfrown t_1{}^\smallfrown t_2{}^\smallfrown\cdots\in\coef(X,(x_n)),$$
$$\epsilon(0)t_0{}^\smallfrown\epsilon(1)t_1{}^\smallfrown\epsilon(2)t_2{}^\smallfrown\cdots\notin\coef(X,(x_n)).$$
This contradicts the unconditionality of $(x_n)$ (cf. Proposition 1.c.6 of \cite{LT}).
\end{proof}

\section{Different Schauder bases of a Banach space}

A question arises naturally:
\begin{quote}
If $(x_n)$ and $(y_n)$ are two Schauder bases of a Banach space $X$, does $E(X,(x_n))\sim_B E(X,(y_n))$?
\end{quote}
Recall that two sequences $(x_n)$ and $(y_n)$ of $X$ are equivalent if $\coef(X,(x_n))=\coef(X,(y_n))$. It is well known that, in every infinite dimensional Banach space with a basis, there exist continuum many mutually non-equivalent normalized bases (see \cite{PS}, 4.1). If we restrict on unconditional bases, only $c_0,\ell_1$ and $\ell_2$ have the property that all unconditional bases are equivalent (see \cite{LT}, Theorem 2.b.10).

When we return to Borel bireducibility between $E(X,(x_n))$ and $E(X,(y_n))$, the question becomes very complicated. From Corollary \ref{fsigma}, if $(x_n)$ is boundedly complete and $(y_n)$ is not, then $E(X,(x_n))\not\sim_BE(X,(y_n))$. In \cite{zippin}, Zippin proved that, for a Banach space $X$ with a basis, $X$ is reflexive iff all bases of $X$ are boundedly complete. So if a non-reflexive space $X$ has a boundedly complete basis, the question turns out to fail for $X$. Which spaces the question can hold for? Perhaps, the most possible candidates might be reflexive spaces. So far, the only known example of such space is Hilbert space $\ell_2$.

\begin{lemma}\label{uptriangle}
Let $(x_n)$ and $(y_k)$ be two bases of a Banach space $X$. If $y_k=\sum_n\alpha_{nk}x_n$ with $\alpha_{nk}=0$ for any $n<k$, then $E(X,(x_n))\sim_BE(X,(y_k))$.
\end{lemma}

\begin{proof}
For any $a\in\R^\N$ and $n\in\N$, define $\theta(a)(n)=\sum_{k\le n}\alpha_{nk}a(k)$. Since $(y_k)$ is a basis of $X$, we can assume that $x_n=\sum_k\beta_{kn}y_k$ for each $n$. Then
$$\delta_{mn}=\left\{\begin{array}{ll}1, & m=n,\cr 0, & m\ne n\end{array}\right.=x^*_m(x_n)=\sum_k\beta_{kn}x^*_m(y_k)=\sum_{k\le m}\beta_{kn}\alpha_{mk}.$$
By induction on $m$, we can prove that $\beta_{kn}=0$ for any $k<n$. Furthermore, we also have $\sum_{n\le m}\alpha_{nk}\beta_{mn}=y^*_m(y_k)=\delta_{mk}$. Therefore, $\theta:\R^\N\to\R^\N$ is invertible and $\theta^{-1}(d)(k)=\sum_{n\le k}\beta_{kn}d(n)$ for any $d\in\R^\N$ and $k\in\N$.

Let $a,b\in\R^\N$. If $a-b\in\coef(X,(y_k))$, then there is $x\in X$ such that $x=\sum_k(a(k)-b(k))y_k$. Note that
$$x^*_n(x)=\sum_k(a(k)-b(k))x^*_n(y_k)=\sum_{k\le n}\alpha_{nk}(a(k)-b(k)).$$
We have $x=\sum_nx^*_n(x)x_n=\sum_n(\theta(a)(n)-\theta(b)(n))x_n$. Thus $\theta(a)-\theta(b)\in\coef(X,(x_n))$. On the other hand, using $\theta^{-1}$, we can prove that, if $\theta(a)-\theta(b)\in\coef(X,(x_n))$, then $a-b\in\coef(X,(y_k))$.

Therefore, $\theta$ and $\theta^{-1}$ witness $E(X,(x_n))\sim_BE(X,(y_k))$.
\end{proof}

\begin{theorem}
For any basis $(y_k)$ of $\ell_2$, we have $E(\ell_2,(y_k))\sim_B\R^\N/\ell_2$.
\end{theorem}

\begin{proof}
For each $n\in\N$, denote $X_n=[y_k]_{k\ge n}$. Find a normalized $x_n\in X_n$ such that $x_n\perp X_{n+1}$. Then we have $\{x_0,\cdots,x_n\}^\perp=X_{n+1}$. We claim that $(x_n)$ is an orthonormal basis of $\ell_2$. It is easy to see that $x_m\perp x_n$ for $m\ne n$. Thus $(x_n)$ is an orthonormal basis of $[x_n]_{n\in\N}$. Let $x\perp[x_n]_{n\in\N}$. Then $x\in\bigcap_n X_n$. Since $(y_k)$ is a basis, let $x=\sum_kr_ky_k$. We can see that $r_k=0$ for any $k$, so $x=0$. It follows that $[x_n]_{n\in\N}=\ell_2$.

Let $\langle\cdot,\cdot\rangle$ be the inner product of $\ell_2$. For any $n,k\in\N$, denote $\alpha_{nk}=\langle y_k,x_n\rangle$. Since $y_k\in X_k$, we have $\alpha_{nk}=0$ for $n<k$. By Lemma \ref{uptriangle}, we have $E(\ell_2,(y_k))\sim_BE(\ell_2,(x_n))=\R^\N/\ell_2$.
\end{proof}

Besides Hilbert space $\ell_2$, we would like to investigate the Borel reducibility between $E(X,(x_n))$'s with $(x_n)$ a basis of $X$. In this section, we focus on two special spaces: $c_0$ and $\ell_1$. Both of them are non-reflexive. $\ell_1$ has boundedly complete bases, while $c_0$ has none.

\begin{theorem}\label{finite}
Let $(x_n)$ be a basis of a Banach space $X$. Let $y_k=\sum_n\alpha_{nk}x_n$ satisfies that, for any $n$, there are only finitely many $k$ such that $\alpha_{nk}\ne 0$. Then
$$E(X,(y_k))\le_B E(X,(x_n))\otimes\R^\N/c_0.$$
\end{theorem}

\begin{proof}
For $m\in\N$, denote $N_m=\max\{N:\forall n\le N\forall k>m(\alpha_{nk}=0)\}$. From the assumption of $\alpha_{nk}$, we can see $\lim_{m\to\infty}N_m=\infty$. Define $\theta_1:\R^\N\to\R^\N$ by $\theta_1(a)(n)=\sum_ka(k)\alpha_{nk}$ for $a\in\R^\N$ and $n\in\N$, and define $\theta_2:\R^\N\to X^\N$ by
$$\theta_2(a)(m)=\sum_{k\le m}a(k)\sum_{n>N_m}\alpha_{nk}x_n$$
for $a\in\R^\N$ and $m\in\N$. The definition of $N_m$ implies that, for $n\le N_m$, $\theta_1(a)(n)=\sum_{k\le m}a(k)\alpha_{nk}$.

Let $a,b\in\R^\N$. We have
$$\begin{array}{ll}&\sum_{k\le m}(a(k)-b(k))y_k\cr
=&\sum_{k\le m}(a(k)-b(k))\left(\sum_{n\le N_m}\alpha_{nk}x_n+\sum_{n>N_m}\alpha_{nk}x_n\right)\cr
=&\sum_{n\le N_m}\left(\sum_{k\le m}(a(k)-b(k))\alpha_{nk}\right)x_n+(\theta_2(a)(m)-\theta_2(b)(m))\cr
=&\sum_{n\le N_m}\left(\theta_1(a)(n)-\theta_1(b)(n)\right)x_n+(\theta_2(a)(m)-\theta_2(b)(m)).\end{array}$$

If $a-b\in\coef(X,(y_k))$, we denote $x=\sum_k(a(k)-b(k))y_k$. Since $(x_n)$ is a basis, we have
$$\begin{array}{ll}x&=\sum_nx_n^*(x)x_n\cr
&=\sum_n\left(\sum_k(a(k)-b(k))x_n^*(y_k)\right)x_n\cr
&=\sum_n\left(\sum_k(a(k)-b(k))\alpha_{nk}\right)x_n\cr
&=\sum_n(\theta_1(a)(n)-\theta_1(b)(n))x_n.\end{array}$$
Thus $\theta_1(a)-\theta_1(b)\in\coef(X,(x_n))$. Furthermore, by $\lim_{m\to\infty}N_m=\infty$, we have
$$x=\lim_{m\to\infty}\sum_{k\le m}(a(k)-b(k))y_k=\lim_{m\to\infty}\sum_{n\le N_m}(\theta_1(a)(n)-\theta_1(b)(n))x_n.$$
Therefore $\lim_{m\to\infty}\|\theta_2(a)(m)-\theta_2(b)(m)\|=0$.

On the other hand, assume that
$$\theta_1(a)-\theta_1(b)\in\coef(X,(x_n))\quad\&\quad\lim_{m\to\infty}\|\theta_2(a)(m)-\theta_2(b)(m)\|=0.$$
Then $\sum_n\left(\theta_1(a)(n)-\theta_1(b)(n)\right)x_n$ is convergent. Furthermore, we have
$$\begin{array}{ll}\sum_k(a(k)-b(k))y_k&=\lim_{m\to\infty}\sum_{k\le m}(a(k)-b(k))y_k\cr
&=\lim_{m\to\infty}\sum_{n\le N_m}(\theta_1(a)(n)-\theta_1(b)(n))x_n\cr
&=\sum_n(\theta_1(a)(n)-\theta_1(b)(n))x_n.\end{array}$$
It follows that $a-b\in\coef(X,(y_k))$.

By Theorem 3.4 of \cite{ding2}, there is a Borel function $\theta':X^\N\to\R^\N$ such that, for $x,y\in X^\N$,
$$\lim_{m\to\infty}\|x(m)-y(m)\|=0\iff(\theta'(x)-\theta'(y))\in c_0.$$
Now we can define the desired Borel reduction of $E(X,(y_k))$ to $E(X,(x_n))\otimes\R^\N/c_0$ as $\theta(a)=(\theta_1(a),\theta'(\theta_2(a)))$ for all $a\in\R^\N$.
\end{proof}

Recall that a basis $(x_n)$ of a Banach space $X$ is called {\it subsymmetric} if it is unconditional and any subsequence $(x_{n_k})$ of $(x_n)$ is equivalent to $(x_n)$ itself. The unit vector bases of $c_0$ and $\ell_p$ are subsymmetric. The following theorem is due to Xin Ma.

\begin{theorem} [\cite{ma}, Theorem 1.1] \label{ma}
Let $(y_n)$ be a basis of a Banach space $Y$, and $(x_n)$ a subsymmetric basis of a closed subspace $X$ of $Y$. Then $E(X,(x_n))\le_B E(Y,(y_n))$.
\end{theorem}

\subsection{Bases in $c_0$}

By Theorem \ref{ma}, among all $E(c_0,(x_n))$ with $(x_n)$ a basis of $c_0$, $\R^\N/c_0$ is the minimum element with respect to Borel reducibility. We are going to find a maximum among them.

We denote
$${\rm cs}=\{a\in\R^\N:\sum_na(n)\mbox{ converges}\}.$$
Let $x^1_n=\sum_{k\le n}e_k$, where $(e_k)$ is the unit vector basis of $c_0$. Then $(x^1_n)$ is a basis of $c_0$ too. Since
$$\sum_na(n)x^1_n=(\sum_{n\ge 0}a(n),\sum_{n\ge 1}a(n),\cdots,\sum_{n\ge k}a(n),\cdots),$$
we can see ${\rm cs}=\coef(c_0,(x^1_n))$.

A simple reduction $\theta(a)=(a(0),-a(0),a(1),-a(1),\cdots)$ witnesses that $\R^\N/c_0\le_B\R^\N/{\rm cs}$. We can easily prove $\R^\N/{\rm cs}\sim_B\R^\N/c$, where $c$ is the set of all convergent sequences. It is worth noting that the unit vectors cannot form a basis of $c$.

For $m\ge 1$, note that $\bigoplus_{i=1}^mc_0\cong c_0$. We choose a suitable basis $(x^m_n)$ of $c_0$ and define ${\rm cs}^{(m)}=\coef(c_0,(x^m_n))$. To do so, let $(e^i_k)$ be the unit vector basis of the $i$-th component space $c_0$, then we set $x^m_{mj+i-1}=\sum_{k\le j}e^i_k$ for $i\le m$ and $j\in\N$.

Furthermore, recall that
$$\left(\bigoplus_{i\in\N}c_0\right)_0=\{(a_n)\in(c_0)^\N:\lim_{n\to\infty}\|a_n\|=0\}.$$
We still have $\left(\bigoplus_{i\in\N}c_0\right)_0\cong c_0$. Now fix a bijection $\langle\cdot,\cdot\rangle:\N^2\to\N$ such that, for any $i$, $\langle i,j\rangle$ is strictly increasing with respect to variable $j$. Let $(e^i_k)$ be the unit vector basis of the $i$-th component space $c_0$ in $\left(\bigoplus_{i\in\N}c_0\right)_0$. We set $x^\infty_n=\sum_{k\le j}e^i_k$ for $n=\langle i,j\rangle$. We can see that $(x^\infty_n)$ is also a basis of $c_0$. Now we denote
$${\rm cs}^{(\infty)}=\coef(c_0,(x^\infty_n)).$$

It is straight forward to check that
$${\rm cs}^{(m)}=\{a\in\R^\N:\forall i\le m(\sum_ja(mj+i-1)\mbox{ converges})\},$$
$${\rm cs}^{(\infty)}=\{a\in\R^\N:\forall j(a(\langle\cdot,j\rangle)\in c_0)\;\&\;\sum_ja(\langle\cdot,j\rangle)\mbox{ converges in }c_0\}.$$

For any Banach space $X$, we define
$${\rm cs}(X)=\{\alpha\in X^\N:\sum_n\alpha(n)\mbox{ converges in }X\},$$
and for any $\alpha\in{\rm cs}(X)$, define
$$\3\alpha\3_X=\sup_n\|\sum_{i\le n}\alpha(i)\|.$$
By Cauchy criterion, it is straightforward to check that $({\rm cs}(X),\3\cdot\3_X)$ is a Banach space. Furthermore, letting $X_n=\{\alpha\in{\rm cs}(X):\forall i\ne n(\alpha(i)=0)\}$, we can see that $(X_n)$ forms a Schauder decomposition of ${\rm cs}(X)$. For any sequence $(x_n)$ in $X$, we claim that
$$E(X,(x_n))\le_BX^\N/{\rm cs}(X).$$
Indeed, for any $a\in\R^\N$ and $k\in\N$, we define $\theta(a)(k)=a(k)x_k$. Then $\theta$ is a Borel reduction of $E(X,(x_n))$ to $X^\N/{\rm cs}(X)$.

An easy observation shows that
$$(\R^m)^\N/{\rm cs}(\R^m)\sim_B\R^\N/{\rm cs}^{(m)},$$
$$(c_0)^\N/{\rm cs}(c_0)\le_B\R^\N/{\rm cs}^{(\infty)}=E(c_0,(x^{(\infty)}_n))\le_B(c_0)^\N/{\rm cs}(c_0).$$
Therefore, $\R^\N/{\rm cs}^{(\infty)}$ is the desired maximum element. Furthermore, we have the following Theorem.

\begin{theorem}
Let $(x_n)$ be a none-zero sequence in $c_0$. Then
$$\R^\N/c_0\le_BE(c_0,(x_n))\le_B\R^\N/{\rm cs}^{(\infty)}.$$
\end{theorem}

\begin{proof}
$E(c_0,(x_n))\le_B\R^\N/{\rm cs}^{(\infty)}$ follows from $(c_0)^\N/{\rm cs}(c_0)\sim_B\R^\N/{\rm cs}^{(\infty)}$. We only need to prove $\R^\N/c_0\le_BE(c_0,(x_n))$.

We may assume that $(x_n)$ is normalized. Note that $c_0^{**}=\ell_\infty$ and the unit ball of $\ell_\infty$ is weak$^*$ compact. There is a subsequence of $(x_n)$ which is weak$^*$ convergent in $\ell_\infty$. Without loss of generality, assume that $(x_n)$ itself is weak$^*$ convergent.

{\sl Case 1.} If $(x_n)$ is not convergent in $c_0$, by the Eberlein-\u Smulian theorem (cf. \cite[p. 41]{diestel}), there is a subsequence $(x_{n_k})$ of $(x_n)$ which is a basic sequence. By Proposition 1.a.11 of \cite{LT}, there is a basic sequence $(y_n)$ in $[x_{n_k}]_{k\in\N}$ which is equivalent to a block basis  $(u_j)$ of the unit vector basis $(e_n)$. We may also assume that $(u_j)$ is normalized. From Proposition 2.a.1 of \cite{LT}, $(u_j)$ is equivalent to $(e_n)$. Let $(x_{n_k}^*)$ be the biorthogonal functionals in $[x_{n_k}]_{k\in\N}^*$. Then $x_{n_k}^*\upharpoonright[y_n]_{n\in\N}\in[y_n]_{n\in\N}^*$. Since $(y_n)$ is equivalent to $(e_n)$, we have $(y_n^*)$ is equivalent to the unit vector basis of $\ell_1$. It follows that $\lim_{n\to\infty}x_{n_k}^*(y_n)=0$ for each $k\in\N$. Then by Proposition 1.a.12 of \cite{LT}, there is a subsequence $(y_{n_j})$ of $(y_n)$ which is equivalent to a block basis of $(x_{n_k})$. Note that $(y_{n_j})$ is still equivalent to the unit vector basis of $c_0$, we have
$$\R^\N/c_0=E(c_0,(y_{n_j}))\le_BE(c_0,(x_{n_k}))\le_BE(c_0,(x_n)).$$

{\sl Case 2.} If $(x_n)$ converges to $x\in c_0$. Then $\|x\|=1$, since $(x_n)$ is normalized. There is a subsequence $(x_{n_k})$ such that $\|x_{n_k}-x\|\le 2^{-k}$. Then for any $a\in\R^\N$, we have
$$\sum_ka(k)x_{n_k}\mbox{ converges}\iff\sum_ka(k)x\mbox{ converges}\iff a\in{\rm cs}.$$
Thus $\R^\N/{\rm cs}\le_BE(c_0,(x_n))$, and hence $\R^\N/c_0\le_BE(c_0,(x_n))$.
\end{proof}

\begin{remark}
Following the proof of the last theorem, we can also get: for any none-zero sequence $(x_n)$ in $\ell_p$ with $p>1$, we have
$$\R^\N/\ell_p\le_BE(\ell_p,(x_n))\quad\mbox{ or }\quad\R^\N/{\rm cs}\le_BE(\ell_p,(x_n)).$$
\end{remark}

\begin{corollary}
Let $(x_k)$ be a none-zero sequence in $c_0$. If for any $n$ there are only finitely many $k$ such that $x_k(n)\ne 0$, then $E(c_0,(x_k))\sim_B\R^\N/c_0$.
\end{corollary}

\begin{proof}
The last theorem implies $\R^\N/c_0\le_BE(c_0,(x_k))$. From Theorem \ref{finite}, we get $E(c_0,(x_k))\le_B\R^\N/c_0\otimes\R^\N/c_0$. So we have $E(c_0,(x_k))\le_B\R^\N/c_0$, since $\R^\N/c_0\otimes\R^\N/c_0\le_B\R^\N/c_0$ is trivial.
\end{proof}

Now we are going to compare $\R^\N/c_0$, $\R^\N/{\rm cs}^{(m)}$, and $\R^\N/{\rm cs}^{(\infty)}$.

Recall that a series $\sum_kx_k$ is said to be {\it perfectly divergent} if for any $\epsilon\in\{-1,1\}^\N$ the series $\sum_k\epsilon(k)x_k$ diverges. The only interesting case is when $\lim_{k\to\infty}x_k=0$. One example in $\ell_p$ is $\sum_n(n+1)^{-1/p}e_n$. Another example in $c_0$ is the follows:
$$\begin{array}{llrrrrrrrrr}
x_0&=(1,&0,&0,&0,&0,&0,&0,&0,&0,&\cdots),\cr
x_1&=(0,&\frac{1}{2},&\frac{1}{2},&0,&0,&0,&0,&0,&0,&\cdots),\cr
x_2&=(0,&\frac{1}{2},&-\frac{1}{2},&0,&0,&0,&0,&0,&0,&\cdots),\cr
x_3&=(0,&0,&0,&\frac{1}{3},&\frac{1}{3},&\frac{1}{3},&\frac{1}{3},&0,&0,&\cdots),\cr
x_4&=(0,&0,&0,&\frac{1}{3},&\frac{1}{3},&-\frac{1}{3},&-\frac{1}{3},&0,&0,&\cdots),\cr
x_5&=(0,&0,&0,&\frac{1}{3},&-\frac{1}{3},&\frac{1}{3},&-\frac{1}{3},&0,&0,&\cdots),\cr
\cdots&\cdots\end{array}$$
In fact, Dvoretzky proved that, in any infinite-dimensional Banach space, there are perfectly divergent series whose general term tends to $0$ (see, e.g. \cite{2kadets}, Theorem 6.2.1). On the other hand, Dvoretzky-Hanani's Theorem shows that for any perfectly divergent series in a finite-dimensional space, its general term does not tend to $0$ (see, e.g. \cite{2kadets}, Theorem 2.2.1).

\begin{lemma}\label{PD}
Let $X$ be a separable infinite-dimensional Banach space and $Y$ a finite-dimensional normed space. Then
$$X^\N/{\rm cs}(X)\not\le_B Y^\N/{\rm cs}(Y).$$
\end{lemma}

\begin{proof}
Let $\sum_kx_k$ be a perfectly divergent series in $X$ with $\lim_{k\to\infty}x_k=0$. We denote $F_n=\{0\}\cup\{x_k:k\le n\}$ for each $n\in\N$, and denote by $E$ the restriction of $X^\N/{\rm cs}(X)$ on $\prod_nF_n$. Assume for contraction that $X^\N/{\rm cs}(X)\le_B Y^\N/{\rm cs}(Y)$. Then we also have $E\le_B Y^\N/{\rm cs}(Y)$.

From Lemma \ref{useful2}, we can find an infinite set $I\subseteq\N$ and an element $\mu\in\prod_{n\notin I}F_n$ such that $E|_I$ (with $\mu$) is additively reducible to $Y^\N/{\rm cs}(Y)$. Without loss of generality, we may assume that $\mu(n)=0$ for $n\notin I$. Therefore, we can find, for each $n\in I$, a natural number $l_n\ge 1$ and a map $H_n:F_n\to Y^{l_n}$ such that the following $\psi$ is a reduction of $E|_I$ to $Y^\N$. Let $(n_k)$ is the strictly increasing enumeration of $I$, then $\psi$ is defined as
$$\psi(\alpha)=H_{n_0}(\alpha(n_0))^\smallfrown H_{n_1}(\alpha(n_1))^\smallfrown H_{n_2}(\alpha(n_2))^\smallfrown\cdots,$$
for any $\alpha\in\prod_{n\in I}F_n$.

For $(y_0,\cdots,y_{l-1})\in Y^l$, denote
$$\3(y_0,\cdots,y_{l-1})\3_Y=\max_{i<l}\|y_0+\cdots+y_i\|.$$
We claim that $\lim_{k\to\infty}\3H_{n_k}(x_k)-H_{n_k}(0)\3_Y=0$. If not, there shall be an infinite $K\subseteq\N$ and $\varepsilon>0$ such that $\3H_{n_k}(x_k)-H_{n_k}(0)\3_Y\ge\varepsilon$ for each $k\in K$. Since $\lim_{k\to\infty}x_k=0$, we can find an infinite $J\subseteq K$ with $\sum_{k\in J}x_k$ converging. Now define $\alpha^\emptyset(n)=0$ for each $n\in I$ and $\alpha^J\in\prod_{n\in I}F_n$ as
$$\alpha^J(n)=\left\{\begin{array}{ll}x_k, & n=n_k,k\in J,\cr 0, & \mbox{otherwise}.\end{array}\right.$$
Then $\sum_{n\in I}(\alpha^J(n)-\alpha^\emptyset(n))=\sum_{k\in J}x_k$ converges, so $(\alpha^J,\alpha^\emptyset)\in E|_I$. Therefore $\psi(\alpha^J)-\psi(\alpha^\emptyset)\in{\rm cs}(Y)$. By Cauchy criterion, we have
$$\lim_{k\to\infty}\3H_{n_k}(\alpha^J(n_k))-H_{n_k}(\alpha^\emptyset(n_k))\3_Y=0.$$
This contradicts with $J\subseteq K$.

By Dvoretzky-Hanani's Theorem, $\sum_k(H_{n_k}(x_k)-H_{n_k}(0))$ is not perfectly divergent, so there is a sequence $\epsilon\in\{-1,1\}^\N$ such that $$\sum_k\epsilon(k)(H_{n_k}(x_k)-H_{n_k}(0))\mbox{ converges.}$$
Now define $\alpha^+,\alpha^-\in\prod_{n\in I}F_n$ as
$$\alpha^+(n_k)=\left\{\begin{array}{ll}x_k, & \epsilon(k)=1,\cr 0, &\epsilon(k)=-1,\end{array}\right.\quad \alpha^-(n_k)=\left\{\begin{array}{ll}0, & \epsilon(k)=1,\cr x_k, &\epsilon(k)=-1.\end{array}\right.$$
We have $\sum_{n\in I}(\alpha^+(n)-\alpha^-(n))=\sum_k\epsilon(k)x_k$ diverges, since $\sum_kx_k$ is perfectly divergent. It follows that $(\psi(\alpha^+)-\psi(\alpha^-))\notin{\rm cs}(Y)$. On the other hand, we can see that
$$\sum_k(H_{n_k}(\alpha^+(n_k))-H_{n_k}(\alpha^-(n_k)))=\sum_k\epsilon(k)(H_{n_k}(x_k)-H_{n_k}(0))\mbox{ converges.}$$
Comparing with $\lim_{k\to\infty}\3H_{n_k}(x_k)-H_{n_k}(0)\3_Y=0$, we have $(\psi(\alpha^+)-\psi(\alpha^-))\in{\rm cs}(Y)$. A contradiction!
\end{proof}

\begin{theorem}
For $m\ge 1$, we have
\begin{enumerate}
\item[(i)] $\R^\N/{\rm cs}^{(m)}\le_B\R^\N/{\rm cs}^{(m+1)}$,
\item[(ii)] $\R^\N/c_0<_B\R^\N/{\rm cs}^{(m)}<_B\R^\N/{\rm cs}^{(\infty)}$.
\end{enumerate}
\end{theorem}

\begin{proof}
Clause (i) is trivial. Already known $\R^\N/c_0\le_B\R^\N/{\rm cs}$, Theorem \ref{unconditional} gives $\R^\N/c_0<_B\R^\N/{\rm cs}$. For proving $\R^\N/{\rm cs}^{(m)}<_B\R^\N/{\rm cs}^{(\infty)}$, we only need to prove
$$(\R^m)^\N/{\rm cs}(\R^m)<_B(c_0)^\N/{\rm cs}(c_0).$$
Since $(\R^m)^\N/{\rm cs}(\R^m)\le_B(c_0)^\N/{\rm cs}(c_0)$ is trivial, the assertion follows from Lemma \ref{PD}.
\end{proof}

\subsection{Bases of $\ell_1$}

Similar to last subsection, we know $\R^\N/\ell_1$ is the minimum element among $E(\ell_1,(x_n))$ with $(x_n)$ a basis of $\ell_1$. Unfortunately, we did not find a maximum for them so far. We managed to find some bases such that the equivalence relations generated by them are not Borel reducible to $\R^\N/\ell_1$.

We denote ${\rm bv}_0={\rm bv}\cap c_0$ where
$${\rm bv}=\{a\in\R^\N:\sum_n|a(n)-a(n+1)|<+\infty\}.$$
Let $y_0^1=e_0$ and $y_n^1=e_n-e_{n-1}$ for $n>0$, where $(e_n)$ is the unit vector basis of $\ell_1$. Then $(y_n^1)$ is a basis of $\ell_1$ too. Since
$$\sum_na(n)y_n^1=(a(0)-a(1),a(1)-a(2),\cdots,a(n)-a(n+1),\cdots),$$
we can see ${\rm bv}_0=\coef(\ell_1,(y_n^1))$.

\begin{theorem}
$\R^\N/\ell_1<_B\R^\N/{\rm bv}_0$.
\end{theorem}

\begin{proof}
For $a\in\R^\N$ and $k\in\N$, define $\theta(a)(2k)=a(k)$ and $\theta(a)(2k+1)=0$. It is clear that $\theta$ is a Borel reduction of $\R^\N/\ell_1$ to $\R^\N/{\rm bv}_0$.

Note that $\|y_0^1+\cdots+y_n^1\|=1$ for any $n\in\N$, but $\sum_ny_n^1$ does not converge. It follows that $(y_n^1)$ is not boundedly complete. By Corollary \ref{fsigma}, we have $\R^\N/{\rm bv}_0\not\le_B\R^\N/\ell_\infty$. And hence $\R^\N/{\rm bv}_0\not\le_B\R^\N/\ell_1$.
\end{proof}

For $m\ge 1$, similar to ${\rm cs}^{(m)}$, note that $\bigoplus_{i=1}^m\ell_1\cong\ell_1$, we choose a suitable basis $(y_n^m)$ of $\ell_1$ such that $\coef(\ell_1,(y_n^m))={\rm bv}_0^{(m)}={\rm bv}^{(m)}\cap c_0$, where
$${\rm bv}^{(m)}=\{a\in\R^\N:\forall i\le m(\sum_j|a(mj+i-1)-a(m(j+1)+i-1)|<+\infty)\}.$$

Also similar to ${\rm cs}^{(\infty)}$, we define ${\rm bv}_0^{(\infty)}$ as follows. Recall that
$$\left(\bigoplus_{i\in\N}\ell_1\right)_1=\{(a_n)\in(\ell_1)^\N:\sum_n\|a_n\|<+\infty\}.$$
We still have $\left(\bigoplus_{i\in\N}\ell_1\right)_1\cong\ell_1$. Fix a bijection $\langle\cdot,\cdot\rangle:\N^2\to\N$ such that, for any $i$, $\langle i,j\rangle$ is strictly increasing with respect to variable $j$. We can find a basis $(y_n^\infty)$ of $\ell_1$ such that  $\coef(\ell_1,(y_n^\infty))={\rm bv}_0^{(\infty)}={\rm bv}^{(\infty)}\cap c_0$, where
$${\rm bv}^{(\infty)}=\{a\in\R^\N:\sum_i\sum_j|a(\langle i,j\rangle)-a(\langle i,j+1\rangle)|<+\infty\}.$$

It is trivial that
$$\R^\N/{\rm bv}_0^{(m)}\le_B\R^\N/{\rm bv}_0^{(m+1)}\le_B\R^\N/{\rm bv}_0^{(\infty)}.$$
We do not know whether they are Borel bireducible with each other. We also compare them with the equivalence relations appear in last subsection.

\begin{theorem}\label{bv}
For any $m\in\N$, we have
\begin{enumerate}
\item[(i)] $\R^\N/\ell_1\not\le_B\R^\N/{\rm cs}^{(m)}$,
\item[(ii)] $\R^\N/c_0\not\le_B\R^\N/{\rm bv}_0^{(\infty)}$,
\item[(iii)] $\R^\N/{\rm bv}_0^{(\infty)}<_B\R^\N/\ell_1\otimes\R^\N/c_0$.
\end{enumerate}
\end{theorem}

\begin{proof}
(i) The proof is combined proofs of Theorem \ref{unconditional} and Lemma \ref{PD}, So we omit some similar arguments.

Since $(\R^m)^\N/{\rm cs}(\R^m)\sim_B\R^\N/{\rm cs}^{(m)}$, we assume for contradiction that $\R^\N/\ell_1\le_B(\R^m)^\N/{\rm cs}(\R^m)$. Denote $F_n=\{i/2^n:i=0,1,\cdots,2^n\}$. From Lemma \ref{useful2}, we can find an infinite set $I\subseteq\N$, a natural number $l_n\ge 1$ and a map $H_n:F_n\to(\R^m)^{l_n}$ for each $n\in I$, satisfying the following requirements. Letting by $(n_k)$ the strictly increasing enumeration of $I$, we define $\psi$ as
$$\psi(a)=H_{n_0}(a(n_0))^\smallfrown H_{n_1}(a(n_1))^\smallfrown H_{n_2}(a(n_2))^\smallfrown\cdots,$$
for any $a\in\prod_{n\in I}F_n$, then we have, for $a,b\in\prod_{n\in I}F_n$,
$$\sum_{n\in I}|a(n)-b(n)|<+\infty\iff(\psi(a)-\psi(b))\in{\rm cs}(\R^m).$$

Choose an $i_n\in\{0,1,\cdots,2^n\}$ for each $n\in I$ such that
$$\lim_{k\to\infty}\frac{i_{n_k}}{2^{n_k}}=0,\quad\sum_{n\in I}\frac{i_n}{2^n}=+\infty.$$
Now denote $x_k=i_{n_k}/2^{n_k}$ for each $k\in\N$. The rest part of proof is almost word for word a copy of the proof of Lemma \ref{PD}.

(ii) By Theorem 8.5.2 and Lemma 8.5.3 of \cite{gaobook}, we only need to prove that ${\rm bv}_0^{(\infty)}$ is ${\bf\Sigma}^0_3$. Define a subset $A\subseteq\R^\N$ satisfying that, for any $a\in\R^\N$,
$$\begin{array}{ll}a\in A &\iff\forall i\in\N(\mbox{there is a subsequence of }a(\langle i,\cdot\rangle)\mbox{ converges to }0)\cr
&\iff\forall i,p,N\in\N\exists j>N(|a(\langle i,j\rangle)|<1/p).\end{array}$$
Note that, for any $a\in{\rm bv}^{(\infty)}$, we have $\lim_{j\to\infty}a(\langle i,j\rangle)$ converges for each $i\in\N$, and $a(\langle i,\cdot\rangle)$ is uniformly convergent to $0$ as $i\to\infty$. It follows that ${\rm bv}_0^{(\infty)}={\rm bv}^{(\infty)}\cap c_0={\rm bv}^{(\infty)}\cap A$. Since ${\rm bv}^{(\infty)}$ is $F_\sigma$ and $A$ is $G_\delta$, we see that ${\rm bv}_0^{(\infty)}$ is ${\bf\Delta}^0_3$. Indeed, ${\rm bv}_0^{(\infty)}$ is a $D_2({\bf\Sigma}^0_2)$ set (for definition of $D_2({\bf\Sigma}^0_2)$, see, e.g., \cite{kechris}, 22.E).

(iii) $\R^\N/{\rm bv}_0^{(\infty)}<_B\R^\N/\ell_1\otimes\R^\N/c_0$ follows from Theorem \ref{finite} and (ii).
\end{proof}

However, we do not know whether $\R^\N/\ell_1\le_B\R^\N/{\rm cs}^{(\infty)}$.

For $p>1$, $(y^1_n)$ is also a sequence in $\ell_p$, but not a basis of $\ell_p$. We have
$$\coef(\ell_p,(y^1_n))=c_0\cap\{a\in\R^\N:\sum_n|a(n)-a(n+1)|^p<+\infty\}.$$
For the Borel reducibility, we claim $E(\ell_p,(y^1_n))\sim_B\R^\N/\ell_p\otimes\R^\N/c_0$. This is because of Theorem \ref{finite} and the following $\theta$ which witnesses $\R^\N/\ell_p\otimes[0,1]^\N/c_0\le_B E(\ell_p,(y^1_n))$. For $a\in\R^\N,b\in[0,1]^\N$ and $n\in\N$, we define
$$\theta(a,b)(n)=\left\{\begin{array}{ll}b(0), & 0\le n\le 3,\cr a(k)+b(k), & n=2^{k+2},\cr
b(k)+\frac{(i-1)(b(k+1)-b(k))}{2^{k+2}-2},& n=2^{k+2}+i,1\le i<2^{k+2}.\end{array}\right.$$
Comparing with Theorem \ref{unconditional}, we are interested in these examples because the unit vector basis of $\coef(\ell_p,(y^1_n))$ is conditional while $\R^\N/\ell_p\otimes\R^\N/c_0$ is generated by an unconditional basis.

\subsection{Rearrangements of bases}

\begin{lemma}\label{rearrange}
Let $(x_n)$ be a basis of a Banach space $X$, $\pi$ a permutation on $\N$. If $(x_{\pi(n)})$ is also a basis, then $E(X,(x_n))\sim_B E(X,(x_{\pi(n)}))$.
\end{lemma}

\begin{proof}
We define $\theta(a)(n)=a(\pi(n))$ for $a\in\R^\N$ and $n\in\N$. Let $a,b\in\R^\N$. If $a-b\in\coef(X,(x_n))$, we denote $x=\sum_n(a(n)-b(n))x_n$. Since $(x_{\pi(n)})$ is also a basis, we have
$$x=\sum_nx_{\pi(n)}^*(x)x_{\pi(n)}=\sum_n(a(\pi(n))-b(\pi(n)))x_{\pi(n)}.$$
So $\theta(a)-\theta(b)\in\coef(X,(x_{\pi(n)}))$. By the same arguments, we can show that $\theta(a)-\theta(b)\in\coef(X,(x_{\pi(n)}))$ implies $a-b\in\coef(X,(x_n))$ too.
\end{proof}

\begin{corollary}
Let $(x_n)$ be an unconditional basis of a Banach space $X$, then for any permutation $\pi$ on $\N$, we have $E(X,(x_n))\sim_B E(X,(x_{\pi(n)}))$.
\end{corollary}

Now we consider rearrangements of the bases $(x^m_n)$ and $(x^\infty_n)$ of $c_0$. Since they are conditional bases, there must be some rearrangements are not bases. However, we always have:

\begin{theorem}
Let $\pi:\N\to\N$ be a permutation. For $m\ge 1$ or $m=\infty$, we have $\R^\N/{\rm cs}^{(m)}\sim_BE(c_0,(x^m_{\pi(n)}))$.
\end{theorem}

\begin{proof}
By Theorem \ref{finite}, we only need to show $\R^\N/{\rm cs}^{(m)}\otimes\R^\N/c_0\le_B\R^\N/{\rm cs}^{(m)}$. We only prove for $m=1$, since other cases are similar. For any $a,b\in\R^\N$ and $n\in\N$, define
$$\theta(a,b)(n)=\left\{\begin{array}{ll}a(k), & n=3k,\cr b(k), & n=3k+1,\cr -b(k), & n=3k+2.\end{array}\right.$$
It is easy to see that $\theta$ is a disired Borel reduction.
\end{proof}

The situation is different when we consider rearrangements of the basis $(y^1_n)$ of $\ell_1$. We only present a special rearrangement of $(y^1_n)$ as:
$$y^1_0,y^1_2,y^1_1,\quad y^1_4,y^1_6,y^1_3,\quad y^1_8,y^1_{10},y^1_5,\quad\cdots.$$
More precisely, define a permutation $\pi_0:\N\to\N$ as follows:
$$\pi_0(n)=\left\{\begin{array}{ll}4k, & n=3k,\cr 4k+2, & n=3k+1,\cr 2k+1, & n=3k+2,\end{array}\right.$$
then we consider the rearranged sequence $(y^1_{\pi_0(n)})$.

\begin{example}
$E(\ell_1,(y^1_{\pi_0(n)}))\sim_B\R^\N/\ell_1\otimes\R^\N/c_0$, i.e., $E(\ell_1,(y^1_{\pi_0(n)}))$ has the highest allowable complexity of Theorem \ref{finite}.
\end{example}

\begin{proof}
It is well known that $\R^\N/c_0\sim_B[0,1]^\N/c_0$. Thus, by Theorem \ref{finite}, we only need to show $\R^\N/\ell_1\otimes[0,1]^\N/c_0\le_B E(\ell_1,(y^1_{\pi_0(n)}))$. For any $a\in\R^\N,b\in[0,1]^\N$, and $j\in\N$, we define
$$\theta(a,b)(j)=\left\{\begin{array}{ll}a(l), & \pi_0(j)=2^{l+1},\cr (2^{l+1}-1)^{-1}, & 2^{l+2}+2\le\pi_0(j)\le 2^{l+2}+2[(2^{l+1}-1)b(l)],\cr 0, &\mbox{otherwise}.\end{array}\right.$$
Let $a_1,a_2\in\R^\N$ and $b_1,b_2\in[0,1]^\N$. Denote $d=\theta(a_1,b_1)-\theta(a_2,b_2)$. Then
$$d\in\coef(\ell_1,(y^1_{\pi_0(n)}))\iff\left\{\begin{array}{l}(d(\pi_0^{-1}(n)))\in{\rm bv}_0={\rm bv}\cap c_0,\cr
\sum_jd(j)y^1_{\pi_0(j)}=\sum_nd(\pi_0^{-1}(n))y^1_n.\end{array}\right.$$
It is easy to see that
$$(d((\pi_0^{-1}(n)))\in c_0\iff d\in c_0\iff a_1-a_2\in c_0,$$
$$\begin{array}{ll}(d(\pi_0^{-1}(n)))\in{\rm bv}&\iff2\sum_l\left(|a_1(l)-a_2(l)|+(2^{l+1}-1)^{-1}\right)<+\infty\cr
&\iff a_1-a_2\in\ell_1.\end{array}$$
Thus
$$(d(\pi_0^{-1}(n)))\in{\rm bv}_0\iff a_1-a_2\in\ell_1.$$
Now suppose $a_1-a_2\in\ell_1$. Then there exists $x\in\ell_1$ such that $x=\sum_nd(\pi_0^{-1}(n))y^1_n$. Note that $\lim_{j\to\infty}d(j)=0$. We have
$$\sum_jd(j)y^1_{\pi_0(j)}=x\iff\lim_{k\to\infty}\sum_{j<3k}d(j)y^1_{\pi_0(j)}=x.$$
By the definition of $\pi_0$,
$$\sum_{j<3k}d(j)y^1_{\pi_0(j)}=\sum_{n<2k}d(\pi_0^{-1}(n))y^1_n+\sum_{k\le i<2k}d(\pi_0^{-1}(2i))y^1_{2i}.$$
If $k=2^{l+1}$ for some $l\in\N$, then
$$\begin{array}{ll}&\|\sum_{2^{l+1}\le i<2^{l+2}}d(\pi_0^{-1}(2i))y^1_{2i}\|\cr
=&2|a_1(l+1)-a_2(n+1)|+2(2^{l+1}-1)^{-1}\left|[(2^{l+1}-1)b_1(l)]-[(2^{l+1}-1)b_2(l)]\right|;\end{array}$$
otherwise, we have $2^{l+1}<k<2^{l+2}$ for some $l\in\N$, then
$$\begin{array}{ll}&\|\sum_{k\le i<2k}d(\pi_0^{-1}(2i))y^1_{2i}\|\cr
\le&\|\sum_{k\le i<2^{l+2}}d(\pi_0^{-1}(2i))y^1_{2i}\|+\|\sum_{2^{l+2}\le i<2k}d(\pi_0^{-1}(2i))y^1_{2i}\|\cr
\le&\|\sum_{2^{l+1}\le i<2^{l+2}}d(\pi_0^{-1}(2i))y^1_{2i}\|+\|\sum_{2^{l+2}\le i<2^{l+3}}d(\pi_0^{-1}(2i))y^1_{2i}\|.\end{array}$$
Hence
$$\sum_jd(j)y^1_{\pi_0(j)}=x\iff\lim_{k\to\infty}\sum_{k\le i<2k}d(\pi_0^{-1}(2i))y^1_{2i}=0\iff b_1-b_2\in c_0.$$
To sum up, $\theta$ is a desired Borel reduction.
\end{proof}

\section{James' space and F.D.D. equivalence relations}

James' space $J$ serves as an example of a non-reflexive space whose double dual is isomorphic to itself. Furthermore, it is also an example of a space with a basis has no unconditional basis (see, e.g., \cite{LT}, Example 1.d.2). In this section, we wish to compare $\R^\N/\ell_p$ and $\R^\N/J$. For this purpose, F.D.D. equivalence relations turn out to be an useful tool.

For $a\in\R^\N$, denote
$$\begin{array}{ll}\|a\|_J=\frac{1}{\sqrt{2}}\sup[&\!\!\!\!\!(a(n_1)-a(n_2))^2+(a(n_2)-a(n_3))^2+\cdots\cr
&+(a(n_{m-1})-a(n_m))^2+(a(n_m)-a(n_1))^2]^{1/2},\end{array}$$
where the supremum taken over all choices of $m$ and $n_1<n_2<\cdots<n_m$. Then James' space defined as
$$J=c_0\cap\{a\in\R^\N:\|a\|_J<+\infty\}.$$
Since the unit vector basis of $J$ is not boundedly complete, by Corollary \ref{fsigma}, $J$ is not $F_\sigma$. Similar to (ii) of proof of Theorem \ref{bv}, we can see that $J$ is ${\bf\Delta}^0_3$. These imply that
$$\R^\N/J\not\le_B\R^\N/\ell_\infty,\quad\R^\N/c_0\not\le_B\R^\N/J.$$

Recall that $\ell^n_p$ is $\R^n$ equipped with the $\ell_p$ norm. By the same sprit, we denote by $J^n$ the $n$-dimensional space equipped with the following norm:
$$\begin{array}{ll}&\|(r_0,\cdots,r_{n-1})\|_J\stackrel{\rm Def}{=}\|(r_0,\cdots,r_{n-1},0,0,\cdots)\|_J\cr
=&\frac{1}{\sqrt{2}}\sup[(r_{n_1}-r_{n_2})^2+\cdots+(r_{n_{m-1}}-r_{n_m})^2+r_{n_m}^2+r_{n_1}^2]^{1/2},\end{array}$$
where the supremum taken over all choices of $m$ and $n_1<n_2<\cdots<n_m<n$. Note that $(J^n)$ is a finite dimensional decomposition of $(\bigoplus_{n\ge 1}J^n)_2$.  Now we consider the F.D.D. equivalence relation $E((\bigoplus_{n\ge 1}J^n)_2,(J^n))$.

Let $(e_{n_k^*})$ be a subsequence of the unit vector basis $(e_n)$ such that, if $k=1+2+\cdots+l$ for some $l\ge1$, then $1+n_{k-1}^*<n_k^*$; otherwise $1+n_{k-1}^*=n_k^*$. The following is one of such examples:
$$e_0,\quad e_2,e_3,\quad e_5,e_6,e_7,\quad e_9,e_{10},e_{11},e_{12},\quad\cdots.$$
It is straightforward to check that there is a canonical Lipschitz isomorphic $\Phi$ from $(\bigoplus_{n\ge 1}J^n)_2$ to the closed linear span $[e_{n_k^*}]_{k\in\N}$ satisfying $\|x\|\le\|\Phi(x)\|_J\le\sqrt{2}\|x\|$ for $x\in(\bigoplus_{n\ge 1}J^n)_2$. Thus
$$E(\left(\bigoplus_{n\ge 1}J^n\right)_2,(J^n))\sim_BE(J,(e_{n_k^*})).$$

Recall that a basis $(x_n)$ of a Banach space $X$ is {\it symmetric} if, for any permutation $\pi:\N\to\N$, $(x_{\pi(n)})$ is equivalent to $(x_n)$. It is well known that every symmetric basis is unconditional and semi-normalized. From Proposition 3.a.3 of \cite{LT}, we know that, for any injection $\sigma:\N\to\N$, $(x_{\sigma(n)})$ is also equivalent to $(x_n)$.

\begin{lemma}\label{cube}
Let $(x_n)$ be a symmetric basis of Banach space $X$. Then
$$E(X,(x_n))\sim_B[0,1]^\N/\coef(X,(x_n)).$$
\end{lemma}

\begin{proof}
Fix a bijection $\langle\cdot,\cdot\rangle:\N\times\Z\to\N$. For any $a\in\R^\N$ and $k,l\in\N$, define
$$\theta(a)(\langle k,l\rangle)=\left\{\begin{array}{ll}0, & a(k)<l,\cr a(k)-l, & l\le a(k)<l+1,\cr 1, & a(k)\ge l+1.\end{array}\right.$$
For any $a,b\in\R^\N$, we split $\N$ into three sets
$$\begin{array}{ll}I_0&=\{k\in\N:[a(k)]=[b(k)]\},\cr
I_1&=\{k\in\N:|[a(k)]-[b(k)]|=1\},\cr
I_2&=\{k\in\N:|[a(k)]-[b(k)]|\ge 2\}.\end{array}$$

For $k\in I_0$, denote $l_k=[a(k)]=[b(k)]$. Then
$$(\theta(a)-\theta(b))(\langle k,l_k\rangle)=a(k)-b(k),$$
while $(\theta(a)-\theta(b))(\langle k,l\rangle)=0$ for $l\ne l_k$.

For $k\in I_1$, denote $l_k=\max\{[a(k)],[b(k)]\}$. Then
$$|(\theta(a)-\theta(b))(\langle k,l_k\rangle)|+|(\theta(a)-\theta(b))(\langle k,l_k-1\rangle)|=|a(k)-b(k)|,$$
while $(\theta(a)-\theta(b))(\langle k,l\rangle)=0$ for $l\ne l_k,l_k-1$.

For $k\in I_2$, we have $|a(k)-b(k)|\ge 1$ and $|(\theta(a)-\theta(b))(\langle k,l\rangle)|=1$ for some $l$.

If $I_2$ is infinite, since $(x_n)$ is semi-normalized, neither $a-b$ nor $\theta(a)-\theta(b)$ is in $\coef(X,(x_n))$. If $I_2$ is finite, without loss of generality, we can assume $I_2=\emptyset$, and both $I_0$, $I_1$ are infinite. Because $(x_n)$ is unconditional, by Proposition 1.c.6 of \cite{LT}, we have
$$\begin{array}{ll}&\sum_k(a(k)-b(k))x_k\mbox{ converges}\cr
\iff&\sum_k|a(k)-b(k)|x_k\mbox{ converges}\cr
\iff&\sum_{k\in I_0}|a(k)-b(k)|x_k,\sum_{k\in I_1}|a(k)-b(k)|x_k\mbox{ are convergent}.\end{array}$$
Still because $(x_n)$ is unconditional, the convergency of $\sum_{k\in I_1}|a(k)-b(k)|x_k$ is equivalent to both
$$\sum_{k\in I_1}|(\theta(a)-\theta(b))(\langle k,l_k\rangle)|x_k,\quad\sum_{k\in I_1}|(\theta(a)-\theta(b))(\langle k,l_k-1\rangle)|x_k$$
are convergent. Their convergency are equivalent to both
$$\sum_{k\in I_1}|(\theta(a)-\theta(b))(\langle k,l_k\rangle)|x_{\langle k,l_k\rangle},\quad\sum_{k\in I_1}|(\theta(a)-\theta(b))(\langle k,l_k-1\rangle)|x_{\langle k,l_k-1\rangle}$$
are convergent, since $(x_n)$ is symmetric and both $k\mapsto\langle k,l_k\rangle$, $k\mapsto\langle k,l_k-1\rangle$ are injection. By the same reason
$$\sum_{k\in I_0}|a(k)-b(k)|x_k\mbox{ converges}\iff\sum_{k\in I_0}|(\theta(a)-\theta(b))(\langle k,l_k\rangle)|x_{\langle k,l_k\rangle}\mbox{ converges}.$$
Since $(\theta(a)-\theta(b))(n)=0$ for any $n$ outside the range of $\langle k,l_k\rangle$ and $\langle k,l_k-1\rangle$, we get
$$\begin{array}{ll}&\sum_k(a(k)-b(k))x_k\mbox{ converges}\cr
\iff&\sum_{n}|(\theta(a)-\theta(b))(n)|x_n\mbox{ converges}\cr
\iff&\sum_{n}((\theta(a)-\theta(b))(n))x_n\mbox{ converges}\end{array}$$
as desired.
\end{proof}

\begin{theorem}\label{symmetric}
Let $(x_n)$ be a symmetric basis of Banach space $X$. Then
$$E(X,(x_n))\le_B\R^\N/J\iff E(X,(x_n))\le_BE(\left(\bigoplus_{n\ge 1}J^n\right)_2,(J^n)).$$
\end{theorem}

\begin{proof}
``$\Leftarrow$'' follows from $E(J,(e_{n_k^*}))\le_B\R^\N/J$. We only prove the ``$\Rightarrow$'' side.

Denote $F_n=\{i/2^n:i=0,1,\cdots,2^n\}$. Since $E(X,(x_n))\le_B\R^\N/J$, from Lemma \ref{useful2}, we can find an infinite set $I\subseteq\N$, a natural number $l_n\ge 1$ and a map $H_n:F_n\to\R^{l_n}$ for each $n\in I$, satisfying the following requirements. Letting by $(n_k)$ the strictly increasing enumeration of $I$, we define $\psi$ as
$$\psi(a)=H_{n_0}(a(n_0))^\smallfrown H_{n_1}(a(n_1))^\smallfrown H_{n_2}(a(n_2))^\smallfrown\cdots,$$
for any $a\in\prod_{n\in I}F_n$, then we have, for $a,b\in\prod_{n\in I}F_n$,
$$\sum_{n\in I}(a(n)-b(n))x_n\mbox{ converges}\iff(\psi(a)-\psi(b))\in J.$$
Now we define, for any $a\in\prod_kF_{n_{2k}}$,
$$\psi'(a)=H_{n_0}(a(0))^\smallfrown(\overbrace{0,\cdots,0}^{l_{n_1}})^\smallfrown H_{n_2}(a(n_2))^\smallfrown(\overbrace{0,\cdots,0}^{l_{n_3}})^\smallfrown\cdots.$$
Set $s_k=l_{n_0}+l_{n_1}+\cdots+l_{n_k}$. Let $(e_{n_k'})$ be the following subsequence of $(e_n)$:
$$e_0,e_1\cdots,e_{s_0-1},\quad e_{s_2},e_{s_2+1},\cdots,e_{s_3-1},\quad e_{s_4},e_{s_4+1},\cdots,e_{s_5-1},\quad\cdots.$$
We still have, for $a,b\in\prod_kF_{n_{2k}}$,
$$\sum_k(a(k)-b(k))x_{n_{2k}}\mbox{ converges}\iff(\psi'(a)-\psi'(b))\in[e_{n_k'}]_{k\in\N}.$$
Therefore, $\psi'$ is a Borel reduction of $\prod_kF_{n_{2k}}$ to $E(J,(e_{n_k'}))$. Note that $[e_{n_k'}]$ is Lipschtz isomorphic to $\left(\bigoplus_kJ^{l_{2k}}\right)_2$, thus Lipschitz embeds into $(\bigoplus_{n\ge 1}J^n)_2$. We get
$$\prod_kF_{n_{2k}}/\coef(X,(x_{n_{2k}}))\le_BE(J,(e_{n_k'}))\le_BE(\left(\bigoplus_{n\ge 1}J^n\right)_2,(J^n)).$$
Since $(x_n)$ is symmetric, $\coef(X,(x_{n_{2k}}))=\coef(X,(x_k))$. Furthermore, since $(x_k)$ is semi-normalized and $\sum_k1/2^{n_{2k}}<+\infty$, we have
$$[0,1]^\N/\coef(X,(x_k))\sim_B\prod_kF_{n_{2k}}/\coef(X,(x_k)).$$
Then Lemma \ref{cube} gives the required result.
\end{proof}

\begin{theorem}
For $p\ge 1$, we have $\R^\N/\ell_p\le_B\R^\N/J\iff p\le 2$.
\end{theorem}

\begin{proof}
Because $\ell_p$ is symmetric, by Theorem \ref{symmetric}, we only need to consider $\R^\N/\ell_p\le_BE((\bigoplus_{n\ge 1}J^n)_2,(J^n))$. Since $\R^\N/\ell_2\le_BE((\bigoplus_{n\ge 1}J^n)_2,(J^n))$ is trivial, and by Dougherty-Hjorth's theorem, $\R^\N/\ell_p\le_B\R^\N/\ell_2$ for any $p\le 2$, we finish the ``$\Leftarrow$'' side.

For proving the ``$\Rightarrow$'' side, suppose $\R^\N/\ell_p\le_BE((\bigoplus_{n\ge 1}J^n)_2,(J^n))$. In Definition 3.2 of \cite{ding2}, the equivalence relation $E((\bigoplus_{n\ge 1}J^n)_2,(J^n))$ was denoted as $E((J^n)_{n\in\N};2)$. It is clear that $E((J^n)_{n\in\N};2)\le_BE(J;2)$, so we have $\R^\N/\ell_p\le_BE(J;2)$. From Theorem 4.8 and Lemma 5.1 of \cite{ding2}, we get $p\le 2$.
\end{proof}

For proving the following theorem, we need a notion of ultraproduct of Banach space. An ultrafilter $\mathfrak A$ on $\N$ is called free if it does not contain any finite set. Let $X$ be a Banach space. Consider the space $\ell_\infty(X)$ of all bounded sequences $\alpha\in X^\N$ with the norm $\|\alpha\|=\sup_n\|\alpha(n)\|$. Its subspace $N=\{\alpha:\lim_{\mathfrak A}\|\alpha(n)\|=0\}$ is closed. The ultraproduct $(X)_{\mathfrak A}$ is the quotient space $\ell_\infty(X)/N$ with the norm $\|(\alpha)_{\mathfrak A}\|_{\mathfrak A}=\lim_{\mathfrak A}\|\alpha(n)\|$. For more details on ultraproducts in Banach space theory, see \cite{heinrich}.

\begin{theorem}
$\R^\N/\ell_2<_BE((\bigoplus_{n\ge 1}J^n)_2,(J^n))<_B\R^\N/J$.
\end{theorem}

\begin{proof}
We only need to prove $E((\bigoplus_{n\ge 1}J^n)_2,(J^n))\not\le_B\R^\N/\ell_2$ and $\R^\N/J\not\le_B E((\bigoplus_{n\ge 1}J^n)_2,(J^n))$. The second one is because that, by Theorem \ref{decomposition}, the equivalence relation $E((\bigoplus_{n\ge 1}J^n)_2,(J^n))$ is Borel reducible to $\R^\N/\ell_\infty$, while $\R^\N/J$ is not.

Assume for contradiction that $E((\bigoplus_{n\ge 1}J^n)_2,(J^n))\le_B\R^\N/\ell_2$. By Theorem 4.8 of \cite{ding2}, $J$ is finitely Lipschitz embeds (see Definition 4.7 of \cite{ding2}) into $\ell_2(\ell_2)\cong\ell_2$. Fix a sequence of finite subsets $(F_n)$ of $J$ such that
$$\{0\}\subseteq F_0\subseteq F_1\subseteq\cdots\subseteq F_n\subseteq\cdots$$
and $\bigcup_nF_n$ is dense in $J$. By the finitely Lipschitz embeddability, there exist $A>0$ and $T_n:F_n\to\ell_2$ satisfying, for $a,b\in F_n$,
$$A^{-1}\|a-b\|_J\le\|T_n(a)-T_n(b)\|\le A\|a-b\|_J.$$
Without loss of generality, we many assume that $T_n(0)=0$ for each $n$. Fix a free ultrafilter $\mathfrak A$ on $\N$. For any $a\in\bigcup_nF_n$, set $m=\min\{n:a\in F_n\}$. Since $\|T_n(a)\|\le A\|a\|_J$ for $n\ge m$, we can define
$$T(a)=(\overbrace{0,\cdots,0}^m,T_m(a),T_{m+1}(a),\cdots,T_n(a),\cdots)_\mathfrak A.$$
By the definition of the norm on $(\ell_2)_\mathfrak A$, it is easy to see that, for any $a,b\in\bigcup_nF_n$,
$$A^{-1}\|a-b\|_J\le\|T(a)-T(b)\|_\mathfrak A\le A\|a-b\|_J.$$
Since $\bigcup_nF_n$ is dense in $J$, $T$ can be extended to a Lipschitz embedding $\tilde T:J\to(\ell_2)_\mathfrak A$. Note that $(\ell_2)_\mathfrak A$ is still a Hilbert space (actually nonseparable, see, e.g., Proposition F.3 of \cite{BL}), so it is reflexive. By Corollary 7.10 of \cite{BL}, $J$ is isomorphic to a closed subspace of $(\ell_2)_\mathfrak A$. This is impossible, because $J$ is not reflexive.
\end{proof}

A well known generalization of James' space is $v_p^0$ for $p>1$ (cf. \cite{pisier}). For $a\in\R^\N$, denote
$$\begin{array}{ll}\|a\|=2^{-1/p}\sup[&\!\!\!\!\!|a(n_1)-a(n_2)|^p+|a(n_2)-a(n_3)|^p+\cdots\cr
&+|a(n_{m-1})-a(n_m)|^p+|a(n_m)-a(n_1)|^p]^{1/p},\end{array}$$
where the supremum taken over all choices of $m$ and $n_1<n_2<\cdots<n_m$. Define $v_p=\{a\in\R^\N:\|a\|<+\infty\}$ and $v_p^0=c_0\cap v_p$. Then $J=v_2^0$.

Similar proof gives
\begin{enumerate}
\item[(i)] Let $(x_n)$ be a symmetric basis of Banach space $X$. Then
$$E(X,(x_n))\le_B\R^\N/v_p^0\iff E(X,(x_n))\le_BE(\left(\bigoplus_{n\ge 1}v_p^n\right)_p,(v_p^n)).$$
\item[(ii)] For $q\ge 1$, we have $\R^\N/\ell_q\le_B\R^\N/v_p^0\iff q\le p$.
\item[(iii)] $\R^\N/\ell_p<_BE((\bigoplus_{n\ge 1}v_p^n)_p,(J^n))<_B\R^\N/v_p^0$.
\end{enumerate}

Furthermore, if we extend the definition of $v_p^0$ to $p=1$, the resulted space is just the space ${\rm bv}_0$. For symmetric basis $(x_n)$ of $X$, we still have
$$E(X,(x_n))\le_B\R^\N/{\rm bv}_0\iff E(X,(x_n))\le_BE(\left(\bigoplus_{n\ge 1}{\rm bv}_0^n\right)_1,({\rm bv}_0^n)).$$
Unlike the previous clause (iii), we have $E((\bigoplus_{n\ge 1}{\rm bv}_0^n)_1,({\rm bv}_0^n))\le_B\R^\N/\ell_1$. A desired Borel reduction $\theta$ defined as, for $\alpha\in\prod_{n\ge 1}{\rm bv}_0^n$,
$$\theta(\alpha)=(\alpha(0),\alpha(1)_0-\alpha(1)_1,\alpha(1)_1,\alpha(2)_0-\alpha(2)_1,\alpha(2)_1-\alpha(2)_2,\alpha(2)_2,\cdots).$$
Therefore, though $\R^\N/\ell_1<_B\R^\N/{\rm bv}_0$, we still have, for symmetric basis $(x_n)$ of $X$,
$$E(X,(x_n))\le_B\R^\N/{\rm bv}_0\iff E(X,(x_n))\le_B\R^\N/\ell_1.$$

\section{Further remarks}

Perhaps the most interesting question is:
\begin{question}
If $(x_n)$ and $(y_n)$ are two bases of a reflexive Banach space $X$, does $E(X,(x_n))\sim_BE(Y,(y_n))$?
\end{question}

Let us denote by $\rm SER$ the set of all Schauder equivalence relations. Ma \cite{ma} showed that, in $\rm SER$, there is a maximum element and $\R^\N/c_0$ is a minimal element with respect to Borel reducibility. Hjorth's dichotomy below $\R^\N/\ell_1$ (see Corollary 5.6 of \cite{hjorth}) implies that $\R^\N/\ell_1$ is another minimal element in $\rm SER$. If we restrict attention on two main classes of spaces $X$ whose unit vector basis is symmetric, i.e., Orlicz sequence spaces and Lorentz sequence spaces, we always have either $\R^\N/\ell_1\le_B\R^\N/X$ or $\R^\N/c_0\le_B\R^\N/X$. Therefore, our first question is:
\begin{question}
Let $(x_n)$ be a symmetric basis of $X$, does either $\R^\N/\ell_1\le_BE(X,(x_n))$ or $\R^\N/c_0\le_BE(X,(x_n))$ hold?
\end{question}
We say two elements $E,F\in{\rm SER}$ are incompatible in $\rm SER$, if no element in $\rm SER$ can be Borel reducible to both $E$ and $F$. It is well known that $\R^\N/\ell_1$ and $\R^\N/c_0$ form an incompatible pair. Ma \cite{ma} also indicated, Farah \cite{farah1} potentially proved that, for any $\alpha$-Tsirelson space $T_\alpha$, $\R^\N/T_\alpha$ are incompatible with either $\R^\N/\ell_1$ or $\R^\N/c_0$, furthermore, whenever $\alpha\ne\beta$, we have $\R^\N/T_\alpha$ and $\R^\N/T_\beta$ are incompatible.

Let $(x_n)$ be an unconditional basis of a Banach space $X$. James \cite{james} proved that, if $(x_n)$ is not boundedly complete, then there exists a block basis $(u_k)$ of $(x_n)$ such that $\coef(X,(u_k))=c_0$ (see also \cite{LT}, Theorem 1.c.10). Comparing with Corollary \ref{fsigma}, we get a dichotomy that, for unconditional basis $(x_n)$ of $X$, exactly one of the following holds:
\begin{enumerate}
\item[(i)] $E(X,(x_n))\le_B\R^\N/\ell_\infty$,
\item[(ii)] $\R^\N/c_0\le_BE(X,(x_n))$.
\end{enumerate}
This dichotomy cannot be generalized to conditional basis, since either $\R^\N/{\rm bv}_0$ or $\R^\N/J$ can serve as a counterexample. This is because clause (i) is equivalent to say that $\coef(X,(x_n))$ is $F_\sigma$, while clause (ii) implies it is ${\bf\Pi}^0_3$-complete, and ${\rm bv}_0$ and $J$ are both $D_2({\bf\Sigma}^0_2)$. Thus we ask two related questions:
\begin{question}
\begin{enumerate}
\item[(I)] Is there a basis $(x_n)$ of $X$ such that $\coef(X,(x_n))$ is ${\bf\Delta}^0_3$ but not $D_2({\bf\Sigma}^0_2)$?
\item[(II)] For any basis $(x_n)$ of $X$, if $\coef(X,(x_n))$ is a ${\bf\Pi}^0_3$-complete set in $\R^\N$, does $\R^\N/c_0\le_BE(X,(x_n))$?
\end{enumerate}
\end{question}

We can use an unconditional basis $(x_n)$ of $X$ to generate an ideal on $\N$. Denote ${\mathcal I}(X,(x_n))=\{A\subseteq\N:\sum_{n\in A}x_n\mbox{ converges}\}$. It is clear that $P(\N)/{\mathcal I}(X,(x_n))\le_BE(X,(x_n))$. Let $(x_n)$ and $(y_n)$ be unconditional bases of $X$ and $Y$ respectively, and suppose $E(X,(x_n))\le_BE(Y,(y_n))$. Applying Lemma \ref{useful2} on $\{0,1\}^\N$, we can find a subsequence $(x_{n_k})$ of $(x_n)$ and a block basis $(u_k)$ of $(y_n)$ such that ${\mathcal I}(X,(x_{n_k}))={\mathcal I}(Y,(u_k))$. Furthermore, for any block basis $(v_k)$ of $(x_n)$, we can find a subsequence of $(v_k)$ and a block basis of $(y_n)$ such that they generate the same ideal. Therefore, we may consider the following question:
\begin{question}
Let $(x_n)$ be an unconditional basis of $X$. To what extent can ideals generated by block bases of $(x_n)$ determine the equivalence relation $E(X,(x_n))$?
\end{question}
It is worth noting that block bases of the unit bases of any space $\ell_p$ (in fact, any Orlicz sequence space) generate the same class of ideals, though they generate totally different equivalence relations with respect to Borel reducibility.

While $(x_n)$ is a conditional basis of $X$, ${\mathcal I}(X,(x_n))$ is not an ideal in general. A powerful alternative tool is ${\mathcal S}(X,(x_n))=\{\epsilon\in\{-1,1\}^\N:\sum_n\epsilon(n)x_n\mbox{ converges}\}$. In fact, it is already used in proofs of theorems \ref{unconditional}, \ref{bv}, and Lemma \ref{PD}.

Besides unconditional bases, we are also interested in H.I. spaces (hereditarily indecomposable Banach spaces). Gowers \cite{gowers} proved that any basis of a Banach space has a block basis which is either unconditional or a basis of an H.I. subspace. Then another interesting question is:
\begin{question}
Let $(x_n)$ be an unconditional basis of $X$, $(y_n)$ a basis of an H.I. space $Y$. Is it possible that $E(X,(x_n))$ and $E(Y,(y_n))$ are compatible in $\rm SER$?
\end{question}
In contract, it is well known that, in this situation, no infinite dimensional Banach space can embed into both $X$ and $Y$.


\subsection*{Acknowledgments}
The main results of this paper were conducted while I visited Nanyang Technological University and Institute for Mathematical Sciences of the National University of Singapore. The article was written while I visited the University of North Texas. I would like to thank Guohua Wu, Chi Tat Chong, and Su Gao for invitations of these visits. I am grateful to Rui Liu, Xin Ma, Zhi Yin, Minggang Yu, and Yukun Zhang for conversations in seminars. Special thanks are due to Su Gao and Bunyamin Sari for useful suggestions.

\end{document}